\documentclass[a4paper,11pt]{article}

\newcommand{\N}{\mathbb{N}}
\newcommand{\Z}{\mathbb{Z}}
\newcommand{\R}{\mathbb{R}}
\newcommand{\C}{\mathbb{C}}

\newcommand{\D}{\mathbb{D}}
\newcommand{\M}{\mathbb{M}}

\newcommand{\bP}{\mathbb{P}}

\newcommand{\Sum}{\displaystyle\sum}
\newcommand{\Int}{\displaystyle\int}

\newcommand\norm[1]{\left\lVert#1\right\rVert}

\newcommand{\Imp}{\mathrm{Im}}

\newcommand{\information}{{
  \bigskip
  \footnotesize
    
    \textbf{Jamerson Bezerra}:
    \textsc{Faculty of Mathematics and Computer Science, Nicolaus Co\-pernicus University, ul. Chopina 12/18, 87-100 Toruń, Poland.} \par\nopagebreak
    \textit{E-mail:} \texttt{jdouglas@impa.br}
  
	\textbf{Pedro Duarte}:
    \textsc{FCUL-Faculdade de Ci\^encias da Universidade de Lisboa, Campo grande 1749-016, Lisboa.} \par\nopagebreak
    \textit{E-mail:} \texttt{pedromiguel.duarte@gmail.com}
}}

\def\supp{\operatorname{supp}}

\def\dim{\operatorname{dim}}

\def\tr{\operatorname{tr}}

\def\max{\operatorname{max}}
\def\min{\operatorname{min}}

\def\spec{\operatorname{Spec}}

\def\quand{\quad\text{and}\quad}

\def\SL{SL}
\def\GL{GL}

\newcommand{\hv}{\hat{v}}
\newcommand{\hw}{\hat{w}}

\newcommand{\Le}{L}
\newcommand{\muRegExp}{\alpha_{\mu}}

\newcommand{\prodSpace}{\Omega}
\newcommand{\prodMeasure}{\tilde{\mu}}
\newcommand{\shift}{\sigma}
\newcommand{\lcCocycle}{\textbf{A}}
\newcommand{\plusStat}{\eta^+}
\newcommand{\minusStat}{\eta^-}

\newcommand{\spaceProb}{\mathcal{P}}
\newcommand{\Prob}{\spaceProb} 

\newcommand{\wconv}{\stackrel{\ast}{\rightharpoonup}}

\newcommand{\ids}{\mathcal{N}}
\newcommand{\schrMatrix}{S}
\newcommand{\towerSymbols}{\Lambda}
\newcommand{\towerKernel}{K}
\newcommand{\towerStat}{\nu}
\newcommand{\prodTower}{\Sigma}

\newcommand{\ellipticDelta}{\delta_0}

\newcommand{\bvec}{\mathbf{v}}

 \newcommand{\Ball}{\mathrm{B}}
\newcommand{\vfrak}{\mathfrak{v}}
\newcommand{\dual}[1]{{#1}^\flat}

\usepackage[margin=1.4in]{geometry}

\usepackage{amsmath,amsthm,amssymb,amsfonts}
\usepackage{blindtext}
\usepackage{epsfig}
\usepackage{multicol}
\usepackage{xcolor}
\usepackage[hidelinks]{hyperref}
\usepackage{graphicx}
\usepackage{color}
\usepackage{psfrag}
\usepackage{multirow}
\usepackage{enumitem}
\usepackage{mathtools}

\usepackage{dirtytalk}
\usepackage[all]{xy}
\usepackage{xfrac}

\usepackage{caption}
\usepackage{subcaption}

\usepackage{titlesec}

\usepackage{enumitem}
\usepackage{float}
\newtheorem{theorem}{Theorem}
\newtheorem{corollary}[theorem]{Corollary}
\newtheorem{proposition}[theorem]{Proposition}

\newtheorem{definition}[theorem]{Definition}
\newtheorem{lemma}[theorem]{Lemma}

\newtheorem{remark}[]{Remark}
\newtheorem{question}[]{Question}
\newtheorem*{question*}{Question}

\newtheorem{ltheorem}{Theorem}

\newtheorem{lcorollary}{Corollary}

\numberwithin{theorem}{section}

\usepackage{tocloft}

\begin{document}

\title{Upper bound on the regularity of the Lyapunov exponent for random products of matrices}
\author{Jamerson Bezerra and Pedro Duarte}
\date{}

\maketitle

\begin{abstract}
    We prove that if $\mu$ is a finitely supported measure on $\SL_2(\R)$ with positive Lyapunov exponent but not uniformly hyperbolic, then the Lyapunov exponent function is not $\alpha$-H\"older around $\mu$ for any $\alpha$ exceeding the Shannon entropy of $\mu$ over the Lyapunov exponent of $\mu$.
\end{abstract}

\setcounter{tocdepth}{1}
\tableofcontents

\section{Introduction}
The law of large numbers, stating that on average an i.i.d. process is close to its theoretical mean, often used to describe the typical statistical behavior of a random sample, is the basis for understanding general additive processes with applications in various branches of mathematics such as probability, combinatorics or ergodic theory.
 
The multiplicative version of the law of large numbers for products of random matrices is the classical theorem of Furstenberg and Kesten~\cite{FKe60}, which asserts that with probability $1$  the logarithmic growth rate of products of random  matrices equals its mean growth rate.
More formally, a special case of this theorem states   that for an i.i.d. sequence of random matrices $L_1, \, L_2, \,\ldots\,$, with common law given by a compactly supported probability  measure $\mu$ on $\GL_d(\R)$, the following asymptotic equality holds almost surely
\begin{align*}
   \lim_{n\to\infty}\frac{1}{n}\log\norm{L_n\ldots L_1}
   = \lim_{n\to\infty}\frac{1}{n}\mathbb{E}[\, \log\norm{L_n\ldots L_1} \, ] ,
\end{align*}
where the right-hand-side, denoted by $\Le(\mu)$,  is the so called \emph{Lyapunov exponent} of the law $\mu$.
The investigation of how the Lyapunov exponent changes as a function of the underlying measure $\mu$  lies at the core of the multiplicative ergodic theory,  with many fundamental contributions during the last 60 years.

The continuity of the Lyapunov exponent $\Le(\mu)$
as a function of the measure $\mu$, with respect to the weak* topology, was established by Furstenberg and Kifer~\cite{FKi83} under a generic  \emph{irreducibility} assumption. A measure $\mu$ is called \emph{irreducible} if there exists no proper subspace of $\R^d$  which is $\mu$-invariant, i.e., invariant under all matrices in the support of $\mu$. Otherwise $\mu$ is called \emph{reducible}, and any proper $\mu$-invariant subspace $S\subset \R^d$
determines the Lyapunov exponent $\Le(\mu\vert_S)$ corresponding to the logarithmic growth rate of the norms
$\norm{(L_n\vert_S)\, \ldots\, (L_1\vert_S)}$.
The continuity of Furstenberg and Kifer actually holds under the weaker \emph{quasi-irreduciblity} assumption. A measure $\mu$ is called \emph{quasi-irreducible} if $\Le(\mu\vert_S)=\Le(\mu)$ for every proper $\mu$-invariant subspace $S\subset \R^d$. See~\cite[Theorem 1.46]{Fur02}.

In~\cite{BV16}, Bocker and Viana proved that for measures supported in $\GL_2(\R)$ the Lyapunov exponent $\mu\mapsto\Le(\mu)$ is continuous  with respect the weak* topology and the Hausdorff distance between their supports. Avila, Eskin and Viana announced that the same result  holds   for measures supported in $\GL_d(\R)$, any $d\geq 2$. See the remark after Theorem 10.1 in~\cite{Viana-book}.

It is then natural to raise the question about the precise modulus of continuity of this map. 

A lower bound for this regularity was provided by Le Page in~\cite{LePage89}.
The Lyapunov exponent is locally H\"older continuous  over an open and dense set of  compactly supported measures on $\GL_d(\R)$, namely the set of quasi-irreducible measures $\mu$ with a gap between the first and second Lyapunov exponents. See also~\cite[Theorem 1]{BaD19}. Recall that a function $E\mapsto f(E)$ is said to be H\"older with exponent $\alpha$, or $\alpha$-H\"older, if there exists a constant $C<\infty$ such that for all $E,E'$,
$$ |f(E)-f(E')|  \leq C\, |E-E'|^\alpha . $$

A function which is H\"older in a neighborhood of each point of its domain is called locally H\"older.
Alternatively, we say that a  function $E\mapsto f(E)$ is point-wisely $\alpha$-H\"older 
if for every $E_0$ there exists a constant
$C<\infty$ and a neighborhood of $E_0$ where for all $E$,
$$ |f(E)-f(E_0)|  \leq C\, |E-E_0|^\alpha . $$
Notice that point-wise  H\"older  is weaker than locally H\"older. In fact the modulus of continuity around a point of
a point-wisely H\"older function can be arbitrary bad.

In~\cite{TV2020} E. Tall and M. Viana  proved that for random $GL_2(\R)$ cocycles the Lyapunov exponents are always point-wisely log-H\"older, and even
point-wisely H\"older when the Lyapunov exponents are distinct.

In the same direction, the quasi-irreducibility hypothesis was discarded in~\cite{DKl20}, where it was established that for finitely supported measures in $\GL_2(\R)$ with distinct Lyapunov exponents, the function $\mu\mapsto \Le(\mu)$ is either locally H\"older or else locally weak-H\"older. Given positive constants $\alpha, \beta\leq 1$, a function $E\mapsto f(E)$ is said to be    $(\alpha,\beta)$-weak H\"older  if there exists a constant $C<\infty$ such that for all $E,E'$,
$$ |f(E)-f(E')|  \leq C\, e^{-\alpha\, \left( \log |E-E'|^{-1}\right)^\beta} . $$
Notice that $(\alpha,1)$-weak H\"older is equivalent to $\alpha$-H\"older.

In the reverse direction, an example due to Halperin~\cite[Appendix 3A]{SiTa85} provides an upper bound on this regularity. The  example consists of the following $1$-parameter family of measures supported in $\SL_2(\R)$,
$$\mu_{a,b,E} :=\frac{1}{2}\, \delta_{A_E} + 
	\frac{1}{2}\, \delta_{B_E},\quad \text{ where  } \quad  A_E=\begin{pmatrix}
		a-E & -1\\ 1 & \phantom{+}0
	\end{pmatrix}, \; 
	B_E=\begin{pmatrix}
		b-E & 0\\ 1 & \phantom{+}0
	\end{pmatrix} .$$
It follows from~\cite[Theorem A.3.1]{SiTa85} (see also Proposition \ref{prop:120522.1}) that the function $E\mapsto \Le(\mu_{a,b,E})$ can not be   $\alpha$-H\"older continuous for  any $\alpha>\frac{2\,\log 2}{\mathrm{arccosh}(1+|a-b|/2)}$.
On the other hand, it is not difficult to see 	that the  measures $\mu_{a,b,E}$ satisfy  the assumptions of Le Page's theorem, which implies that the function $E\mapsto \Le(\mu_{a,b,E})$
 is indeed H\"older continuous, but with a very small H\"older exponent $\alpha$ when $a-b$ is large.
 
In the same spirit, in~\cite{DKS19}, the authors provide the following example where the Lyapunov exponent is not even weak-H\"older continuous. They consider the measure
$$\mu :=\frac{1}{2}\, \delta_{A} + 
	\frac{1}{2}\, \delta_{B},\quad A=\begin{pmatrix}
		0 & -1\\ 1 & \phantom{+}0
	\end{pmatrix}, \; B=\begin{pmatrix}
		e & 0\\ 0 & e^{-1}
	\end{pmatrix}$$
and prove that there exists a curve $\tilde \mu_t=\frac{1}{2}\,\delta_{A_t} +\frac{1}{2}\,\delta_{B_t}$ through $\tilde \mu_0=\mu$ such that $t\mapsto L_1(\tilde\mu_t)$ is not  weak-H\"older  around $t=0$. Notice that $\mu$ is not quasi-irreducible and $\Le(\mu)=0$ so that $\mu$ does not satisfy any of the assumptions of Le Page's theorem.

In contrast with this low regularity, a classical theorem of  Ruelle proves the  analiticity of the Lyapunov exponent for  uniformly hyperbolic measures with $1$-dimensional unstable direction, see~\cite[Theorem 3.1]{Rue79}. A compactly supported measure $\mu$ on $\GL_d(\R)$ is said to be \emph{uniformly hyperbolic} if the linear cocycle generated by $\mu$  is uniformly hyperbolic (see Section \ref{260122.5}).

From now on we focus on the class of finitely supported measures in   $\SL_2(\R)$, where the lack of regularity of the Lyapunov exponent can only occur outside of the class of uniformly hyperbolic measures.
In~\cite{ABY10}, Avila, Bochi and Yoccoz gave a characterization of the uniformly hyperbolic cocycles generated by a finitely supported measure in $\SL_2(\R)$  in terms of  existence of an invariant \emph{multicone}. With this characterization they prove that the complement of the closure of the uniformly hyperbolic measures is the set of \emph{elliptic} measures, meaning the finitely supported measures such that the semigroup $\Gamma_{\mu}$ generated by the support of $\mu$ contains a \emph{elliptic element}, i.e., a matrix conjugated to a rotation.

Given a hyperbolic matrix $A\in\SL_2(\R)$
we denote by $\hat s(A)$, respectively $\hat u(A)$, the stable direction, respectively the unstable direction of $A$ in the projective space $\bP^1$. We say that $\mu$ has a heteroclinic tangency if there are matrices $A,B,C\in\Gamma_\mu$ such that $A$ and $B$ are hyperbolic and $C\, \hat u(B)=\hat s(A)$. In this case we also say that $(B,\, C,\, A)$ is a tangency for $\mu$. If moreover $A=B$, we say that $\mu$ has a homoclinic tangency.
Heteroclinic tangencies are referred to  as heteroclinic connections in~\cite{ABY10}\footnote{In this work   Avila, Bochi and Yoccoz characterize the boundary of uniformly hyperbolic cocycles. By~\cite[Remark 4.2]{ABY10}, heteroclinic connections can never be homoclinic connections for cocycles at the boundary of the uniformly hyperbolic ones.}.
If $\mu$ is not uniformly hyperbolic but $\Le(\mu)>0$, i.e., if $\mu$ is non-uniformly hyperbolic, then $\Gamma_\mu$ contains hyperbolic matrices. By Theorem 4.1 of~\cite{ABY10}, in this case the semigroup $\Gamma_\mu$ contains either a heteroclinic tangency or else a non hyperbolic matrix, i.e., an elliptic or parabolic matrix. In each of these two cases we can produce heteroclinic tangencies with an arbitrary small perturbation. See Proposition
~\ref{heteroclinic tangencies  existence}. 
Hence measures with heteroclinic tangencies are dense in the class of non-uniformly hyperbolic measures.

\subsection{Results}
Let $H(\mu)$ be the \emph{Shannon's entropy} (see Section \ref{sec:260722.1}) of the finitely supported measure $\mu$.
\begin{ltheorem}\label{mainThm}
    Let $\mu$ be a finitely supported measure on $\SL_2(\R)$. Assume that $\Le(\mu)>0$, $\mu$ is irreducible and that $\mu$ has a heteroclinic tangency. Then, there exists an analytic one parameter family of finitely supported measures $\{\mu_E\}_E$ such that $\mu_0 = \mu$ and for any $\alpha > {H(\mu)}/{L(\mu)}$, the function $E\mapsto L(\mu_E)$ is not locally $\alpha$-H\"older at any neighborhood of $E=0$.
\end{ltheorem}

\begin{remark}
Our result implies a similar conclusion as in Halperin/Simon-Taylor example with a less  sharper threshold. For simplicity we consider  the parameters $a=0$ with energy $E=0$. In this example
$H(\mu_{0,b,0})=\log 2$ while 
$$  \Le(\mu_{0,b,0})\leq \frac{1}{2}\,\log\norm{B_0 }
= \frac{1}{2}\, \log\sqrt{\frac{2+b^2 + |b| \sqrt{b^2+4 }}{2}}
< \frac{1}{2}\, \mathrm{arccosh}\left(1+\frac{|b|}{2}\right) .$$
The last two quantities are asymptotically equivalent,
which implies that 
$$ \frac{H(\mu_{0,b,0})}{\frac{1}{2}\, \log \norm{B_0}}\sim \frac{2\,\log 2}{\mathrm{arccosh}(1+|b|/2)}  \quad \text{ as } \; b\to\infty. $$
\end{remark}

Set
\begin{align*}
    \muRegExp := \sup\left\{
        \alpha>0 \colon \ \Le \text{ is locally } \alpha\text{-H\"older around } \mu
    \right\}.
\end{align*}
\begin{lcorollary}\label{Cor:010722.8}
    Let $\mu$ be a finitely supported measure on $\SL_2(\R)$ with $L(\mu)>0$. Then, either $\mu$ is uniformly hyperbolic and $L$ is locally analytic around $\mu$, or else
    \begin{align*}
        \alpha_{\mu}\leq \frac{H(\mu)}{L(\mu)}.
    \end{align*}
\end{lcorollary} 

As a consequence of the proof of \ref{mainThm} we have the following application in mathematical physics (for precise definitions see Section \ref{sec:050722.1}).
\begin{lcorollary}\label{cor:070122.9}
    Consider the Anderson model of the discrete Schr\"odinger operators associated with a finitely supported measure $\mu$. Let $\alpha > \frac{H(\mu)}{L(\mu)}$ and $E_0$ be an energy in the spectrum. Then, the integrated density of states function $E\mapsto \mathcal{N}(E)$ and the Lyapunov exponent function $E\mapsto \Le(E)$ are not $\alpha$-H\"older continuous at any neighborhood of $E_0$.
\end{lcorollary}

\subsection{Relations with other dimensions}

See Section \ref{sec:260722.1} for a precise description of the objects treated in this subsection.

The study of formulas relating (some type of) dimension, entropy and Lyapunov exponent has a vast history with many contributions in different settings (see for instance \cite{LY1984} and \cite{LY1985} for diffeormorphisms of a compact manifold and \cite{Ba2015} for self affine measures).

For $\SL_2(\R)$ supported measures $\mu$ with $\Le(\mu)>0$, Ledrappier in \cite{Le1983}, proved that we have a dimension type formula for any (forward) stationary measures $\eta$ associated with $\mu$, namely
\begin{align}\label{eq:020622.2}
    \text{Dim }\eta = \min\left\{
        1,\, \frac{h_F(\eta)}{2L(\mu)}
    \right\}.
\end{align}
$\text{Dim}$ is a different notion of dimension from $\dim$ given in Section \ref{sec:260722.1} (See \cite[Remark 2.34]{Fur02}), but in the case that $\eta$ is exact dimensional they coincide. The exactness of the dimension of the stationary measures was established by Hochman and Solomyak in \cite{HoSo2017} assuming additionally that $\mu$ is irreducible. In particular, if $\eta^+$ and $\eta^-$ denote respectively the forward and backward stationary measures then
\begin{align*}
    \dim\eta^{\pm} = \min\left\{
        1,\, \frac{h_F(\eta^{\pm})}{2\Le(\mu)}
    \right\} \leq \min\left\{
        1,\, \frac{H(\mu)}{2\Le(\mu)}
    \right\}.
\end{align*}
Moreover, in \cite{HoSo2017} they provided, among other things, conditions to obtain Ledrappi\-er-Young type formulas relating the dimension of the stationary measure, the entropy and the Lyapunov exponents, i.e., 
\begin{align*}
    \dim\eta^{\pm}
    = \min\left\{
        1,\, \frac{H(\mu)}{2\Le(\mu)}
    \right\},
\end{align*}
where $h_F(\eta^{\pm})$ is the \emph{Furstenberg entropy} of $\eta^{\pm}$. In light of the above discussion we leave the following questions.
\begin{question}
    Assume that $\alpha > \dim\eta^+ + \dim\eta^-$. Under the assumptions of Theorem \ref{mainThm}, is it true that the Lyapunov exponent is not $\alpha$-H\"older continuous in any neighborhood of $\mu$?
\end{question}

\begin{question}
    Is $\frac{H(\mu)}{\Le(\mu)}$ a sharp bound for the regularity? In other words, is there an example where $\alpha_{\mu} = \frac{H(\mu)}{\Le(\mu)}$?
\end{question}
Halperin's example above does not answer this question.

\begin{question}
    In the case that $\frac{H(\mu)}{\Le(\mu)} \geq 1$, is it true that the Lyapunov exponent is Lipschitz continuous function around $\mu$?
\end{question}

\begin{question}
    Is it possible to express the lower bound for the regularity in terms of some of the previous measurements?
\end{question}

\subsection{Sketch of the proof and organization}

In Mathematical Physics the Thouless formula~\eqref{190122.3} relates the Lyapunov exponent of a Schr\"odinger cocycle with the integrated density of sates (IDS) of the corresponding Schr\"odinger operator. It follows from this identity~\eqref{190122.3} that the Lyapunov exponent and the IDS,  as functions of the energy, share the same modulus of continuity. See Proposition~\ref{prop:120522.1}. The IDS is a spectral quantity that measures the asymptotic distribution of the eigenvalues of  truncation matrices of the Schr\"odinger operator as the  size of the truncation tends to infinity.
The strategy to break the  H\"older regularity of the IDS in Halperin's example 
is to establish around a certain energy a very large concentration of eigenvalues of the Schr\"odinger truncated matrices which  implies a disproportionately large leap of the IDS around that energy,  see~\cite[Appendix 3]{SiTa85}.
Then, as explained above, the loss of H\"older regularity passes from the IDS to the Lyapunov exponent.

Let $\mu$ be an irreducible and  finitely supported measure on $\SL_2(\R)$ with positive Lyapunov exponent and  $\lcCocycle:\prodSpace\to \SL_2(\R)$ be the associated locally constant cocycle. 
In order to use the strategy described  above, in Section~\ref{250122.1} we embed the cocycle $\lcCocycle$ into a  family of locally constant Schr\"odinger cocycles  over a Markov shift.

We call \emph{matching} to a configuration where 
the horizontal direction $e_1=(1,0)$ is mapped in $n$ iterations to the vertical direction $e_2=(0,1)$ in a way that that   $e_1$ is greatly expanded in the first half iterations followed by  a  similar contraction in the second half iterations. The number $n$ is referred to as the \emph{size} of the matching. A matching of size $n$ at some energy $E_0$ determines an almost eigenvector for an $n\times n$ truncated Schr\"odinger matrix, which then implies a true nearby  eigenvalue $E_0^\ast \approx E_0$  of the same matrix.
Hence matchings of size $n$ for energies in some small interval $I$ can be used to count eigenvalues
of a truncated Schr\"odinger operator of size $n$.

By Proposition~\ref{prop:solvingEquations}, a heteroclinic tangency of the cocycle $\lcCocycle=\lcCocycle_{(0)}$ implies many nearby matchings of any chosen large size $n$, spreading through a small interval of length $\sim e^{-c\, n}$.
Because these matchings are still not enough to break  the H\"older regularity in the stated form, 
we prove in Proposition~\ref{prop:010722.1} that a single tangency will cause many more  tangencies to occur at nearby energies, which are in some sense typical.
Propositions~\ref{prop:solvingEquations} and~\ref{prop:010722.1}  were designed to be used recursively in the sense that the output of the second feeds the input of the first. They could be used recursively  to characterize the fractal structure of matchings and tangencies, a path we do not explore in this work. We do  use them in a single cycle to gather the matchings, of some appropriate size, associated to a typical nearby heteroclinic tangency.
The matchings coming from a typical tangency are now 
 enough to break  the H\"older regularity in the stated form.

Proposition~\ref{prop:010722.3} plays a key role in the proof of Theorem A, to estimate the number of matchings and tangencies from Propositions~\ref{prop:solvingEquations} and~\ref{prop:010722.1}.
On the other hand the proof of Proposition~\ref{prop:010722.3} relies on a  characterization of the projective random walk  distribution in Proposition~\ref{P1}
and a few Linear Algebra facts on the geometry of the projective action in Appendix~\ref{appendix:LA}. See propositions~\ref{balanced radius},~\ref{lem:080722.3} and
Lemma~\ref{190722.10}.

\paragraph{Organization:}
This work is organized as follows. Section \ref{sec2:010622.1} contains the general definitions that will be used throughout the paper. In section \ref{sec3:010622.2} we define and state some properties of locally constant linear cocycles and Furstenberg measures. We discuss general spectral properties of Schr\"odinger operators in Section \ref{sec4:010622.3} and in Section \ref{250122.1} we show how to embed a general locally constant cocycle into a Schr\"odinger family over a Markov shift. In Section \ref{sec6:010622.5} we obtain a lower bound for the oscillation of the integrated density of states in terms of counting matchings. Section \ref{sec7:010622.6} contains the core technical results of the work, namely propositions \ref{prop:010722.3}, \ref{prop:solvingEquations} and \ref{prop:010722.1}. Section \ref{sec8:010622.7} provides lower bounds for the measure of the set of matchings. In Section \ref{sec9:010622.8} we give the proof of the results. The Appendix \ref{appendix:LA} contains the linear algebra tools needed in this work and Appendix \ref{Appendix:derivativeProjectiveActions} describes some of the formulas for derivatives of projective actions.


\paragraph{Logical structure:}The following picture describes the logical structure of the proof of Theorem \ref{mainThm}.
\begin{figure}[h]
    \centering
    \includegraphics[width=0.7\textwidth]{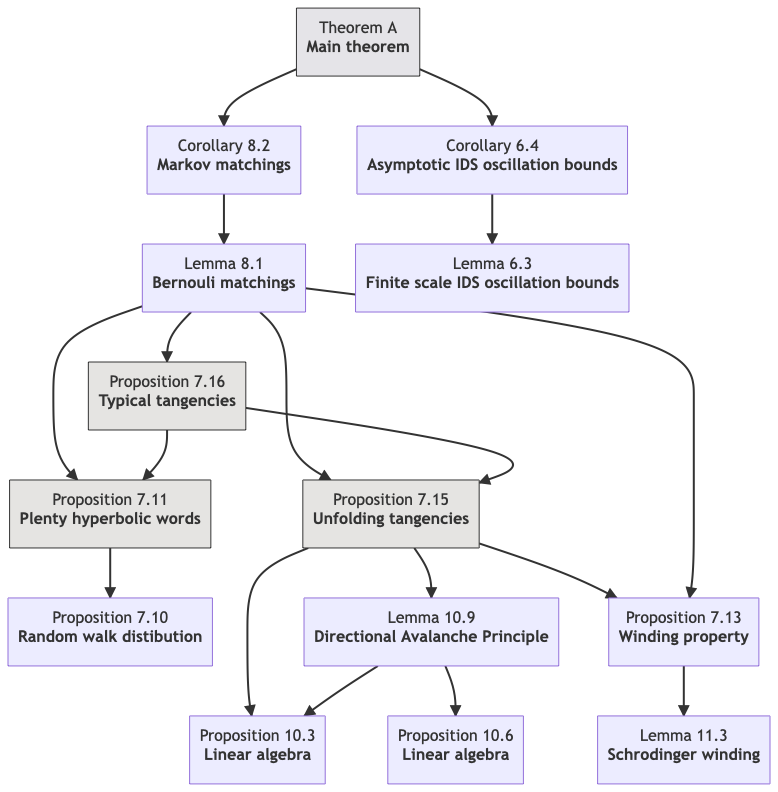}
    \caption{Logical dependencies}
    \label{EntropyVersion}
\end{figure}

\paragraph{Acknowledgments:}
Both authors were supported by FCT-Funda\c{c}\~{a}o para a Ci\^{e}ncia e a Tecnologia through the project  PTDC/MAT-PUR/29126/2017.
J. B. was also supported by the Center of Excellence "Dynamics, Mathematical Analysis and Artificial Intelligence" at Nicolaus Copernicus University in Torun.
P.D.  was also supported by CMAFCIO through FCT project  UIDB/04561/2020.

\section{Basic definitions and general concepts}\label{sec2:010622.1}
In this subsection we establish some of the general notation used throughout this work.

\subsection{Preliminary definitions and notations}
\label{definitions and notations}
\begin{itemize}
\item 
We denote by $\GL_d(\R)$ and $\SL_d(\R)$ respectively the group of $d\times d$ invertible matrices and its subgroup of matrices with determinant one. Given a $d\times d$ square matrix $H$ we denote  its spectrum by $\spec(H)$ and by $|\spec(H)|$   the number of elements in $\spec(H)$ counted with multiplicity. Unless otherwise stated, $\norm{H}$ refers to the operator norm of the matrix $H$.

\item 
The projective space of $\R^2$, consisting of all lines in $\R^2$,  is denoted by $\bP^1$. Its  points are denoted by $\hat v\in \bP^1$.
After introducing a projective point $\hat v$, by convention the letter $v$ will stand for any unit vector aligned with the line $\hat v$.
A natural distance in $\bP^1$ is given by
$d(\hat v,\, \hat w):=|v\wedge w|=\sin \measuredangle(\hat v, \hat w)$.

A  matrix $A\in\SL_2(\R)$ induces a projective automorphism $\hat A\colon \bP^1\to \bP^1$, where $\hat A\, \hat v:=\widehat{A\, v}$
is the line determined by the unit vector $A\, v/\norm{A v}$. For the sake of notational simplicity we often write $A\,\hat v$ instead of $\hat A\, \hat v$.

\item 
We use the standard classification for $\SL_2(\R)$ matrices as \emph{elliptic}, $\emph{parabolic}$ or \emph{hyperbolic} meaning respectively that the absolute value of the trace is smaller than one, equal or greater than two.

\item 
Let $X$ be a compact metric space.
The space of all Borel probability measures on $X$ is denoted by $\Prob(X)$. This is a convex and compact set with respect to the weak* topology. Given a sequence of measures $\eta_n\in\Prob(X)$, we say that $\eta_n$ converges weak* to $\eta$ in $\Prob(X)$,
and write  $\eta_n\wconv \eta$, if for every continuous function $\varphi\in C^0(X)$,
$$\int \varphi\, d\eta=\lim_{n\to\infty} \int \varphi\, d\eta_n.$$

\item 
Given two probability measures $\mu_1, \mu_2\in \Prob(\SL_2(\R))$, the convolution between $\mu_1$ and $\mu_2$ is the measure
\begin{align*}
    \mu_1*\mu_2 := \int_{\SL_2(\R)} g_*\mu_2\, d\, \mu_1(g).
\end{align*}
The $n$-th convolution power, $\mu^{*n}$, of a measure $\mu\in \SL_2(\R)$ is defined inductively by $\mu^{*n} := \mu^{*(n-1)}*\mu$.

    \item 
    Let $\Lambda$ be a finite set and $\Sigma=\Lambda^\Z$. Given a finite word
    $a=(a_0,a_1,\ldots, a_{m-1})\in\Lambda^m$ and $k\in\Z$ the set
    $$ [k;\, a]:=\left\{ \zeta\in\Sigma\, \colon \, \zeta_{j+k}=a_j,\, \forall j=0,1,\ldots, m-1 \, \right\}$$
    is called the \emph{cylinder} of $\Sigma$ determined by the word $a$ and the position $k$. The integer $m$ is referred to as the length of the cylinder.

    \item  Given a compact metric space $(X,d)$, $0<\theta<1$ and a continuous function $\varphi\in C^0(X)$,
    the $\theta$-H\"older constant of $\varphi$ is defined by
    $$ v_\theta(\varphi):=\sup_{\substack{x,y\in X\\x\neq y}}\frac{|\varphi(x)-\varphi(y)|}{d(x,y)^\theta} . $$
    The space of $\theta$-H\"older continuous functions on $X$ is
    $$C^\theta(X):=\left\{ \varphi\in C^0(X)\, \colon \, v_\theta(\varphi)<\infty \, \right\} $$
    which endowed with the norm
    $$ \norm{\varphi}_\theta:= \norm{\varphi}_\infty + v_\theta(\varphi) $$
    becomes a Banach algebra.
    
    \item 
    Given sequences of real numbers 
    $(a_n)_n$ and $(b_n)$ with $a_n,\, b_n>0$ we write
    \begin{itemize}
        \item $a_n = O(b_n)$ if there exists an absolute constant $C>0$ and $n_0\in\N$ such that $a_n \leq C\, b_n$ for every $n\geq n_0$;
        
        \item $a_n\lesssim b_n$ or $b_n \gtrsim a_n$ if $a_n = O(b_n)$;
                
        \item $a_n\sim b_n$ if $\lim_{n\to\infty} a_n/b_n  = 1$.
        
        \item For $\gamma,\, t>0$ we write $\gamma\, \gg\, t$ to indicate that $t$ is much smaller than $\gamma$.
    \end{itemize}

    \item Given some interval  $J\subset \R$ or $J\subseteq \bP^1$, we denote by $|J|$ the length (size) of $J$. Given a positive number $t$, we denote by $t\, J$ the interval with the same center as $J$ and size $t\, |J|$.
\end{itemize}

\subsection{Linear cocycles}
Let $X$ be a compact metric space, with Borel $\sigma$-algebra $\mathcal{B}$. Consider a homeomorphism $T:X\to X$  which preserves a probability measure $\xi$ defined on $\mathcal{B}$ and such that the system $(T,\xi)$ is ergodic. The triple $(X, T, \xi)$ is referred to as the \emph{base dynamics}.

Any continuous map $A:X\to \SL_2(\R)$ defines a \emph{linear cocycle}, over the base dynamics $(X,T,\xi)$, $F_A:X\times \R^2\to X\times\R^2$ given by $F_A(x,v) := (Tx, A(x)\, v)$. By linearity of the fiber action we can define the projectivization of  $F_A$ as the map $\hat{F}_A:X\times \bP^1\to X\times\bP^1$ given by $\hat{F}_A(x,\hat v) := (Tx, \hat A(x)\, \hat v)$. We also use the term \emph{linear cocycle} referring to the map $A:X\to \SL_2(\R)$ when the base dynamics is fixed.

Note that for each $n\in \Z$, the $n$-th iteration of the linear cocycle $F_A$ sends a point $(x,v)\in X\times\R^2$ to $(T^nx, A^n(x)\, v)$, where
\begin{align*}
    A^n(x) = \left\{
        \begin{array}{lc}
            A(T^{n-1}x)\,\ldots\, A(Tx)\, A(x)  & \text{if } n>0 \\
            \; I  & \text{if } n=0 \\
            A(T^nx)^{-1}\,\ldots\, A(T^{-2}x)^{-1}\, A(T^{-1}x)^{-1}  & \text{if } n<0.
        \end{array}
    \right.
\end{align*}

The \emph{Lyapunov exponent} of the cocycle $F_A$ can be defined as the limit
\begin{align*}
    L(A) = \lim_{n\to\infty}\frac{1}{n}\log\norm{A^n(x)},
\end{align*}
which exists and is constant for $\xi$-a.e. $x\in X$ as a consequence of Kingman's sub-additive ergodic theorem. Notice that the Lyapunov exponent depends on the base dynamics despite the fact that the notation $L(A)$ does not   refer to  $(X,T,\xi)$. The underlying base dynamics should always be clear from the context.

\subsection{Transition kernel and stationary measures}
Let $X$ be a compact metric space. We call  \emph{transition kernel} to any continuous map $K:X\to \Prob(X)$. Any transition kernel $K$ induces a linear operator $K:C^0(X)\to C^0(X)$
\begin{align*}
    (K\varphi)(x) := \int_X \varphi dK_x,
\end{align*}
acting on the space $C^0(X)$ of  continuous functions $\varphi:X\to \R$. This is called the \emph{Markov operator} associated to the transition kernel $K$. The adjoint of $K$, $K^\ast$, in the space of probability measures $\Prob(X)$ is given by
\begin{align*}
    K^\ast\xi := \int K_x d\xi(x).
\end{align*}
We say that a probability measure $\xi_0$ is \emph{stationary} for $K$ if $\xi_0$ is a fixed point of $K^\ast$, i.e., if for every  $\varphi\in C^0(X)$,
\begin{align*}
    \int_X\varphi \, d\xi_0 = \Int_X\left(
        \int_X \varphi(y)\ dK_x(y)
    \right)\, d\xi_0(x).
\end{align*}

Consider the process $e_n\colon X^\N\to X$,
$e_n(\omega):= \omega_n$.
Given a probability measure $\xi\in \Prob(X)$ there exists a unique measure $\tilde{\xi}$ in the space $X^{\N}$ such that:
\begin{enumerate}
    \item[(a)] $\tilde \xi (e_0^{-1} (E) )=\xi(E)$, \; $\forall\, E\in \mathcal{B}$;
    \item[(b)] $\tilde \xi ( e_n^{-1} (E) \, \vert \, e_{n-1}=x )= K_{x}(E)$, \; $\forall\, E\in \mathcal{B},\, \forall\, x\in X$.
\end{enumerate}
We say that the measure $\tilde{\xi}$ is the \emph{Kolmogorov extension} of the pair $(K,\xi)$. The following statements are equivalent:

\begin{enumerate}
    \item[(1)]   $\xi$ is $K$-stationary,
    \item[(2)]   $\tilde \xi$ is  invariant under the one-sided shift map $T\colon X^{\N}\to X^{\N}$,
$T(x_n)_{n\in\N}:=(x_{n+1})_{n\in \N}$,
    \item[(3)]   $e_n:X^\N\to X$ is a stationary Markov process with  transition kernel $K$ and common law $\xi$.
\end{enumerate}
When these conditions hold we refer to $(K,\xi)$ as a \emph{Markov system}.
In this case the Kolmogorov extension $\tilde \xi$ admits a natural extension
to $X^\Z$, still denoted by $\tilde\xi$, for which the two sided process 
$e_n\colon X^\Z\to X$,
$e_n(\omega):= \omega_n$, is a stationary Markov process. Moreover $\tilde \xi$ is invariant under the two sided shift map $T\colon X^{\Z}\to X^{\Z}$,
\begin{align*}
    T(\ldots,x_{-1}, \mathbf{x_0}, x_1, x_2,\ldots) := (\ldots, x_0, \mathbf{x_1}, x_2,\ldots),
\end{align*}
where the bold term in the above expression indicates the $0$-th position of the sequence.  The dynamical system $(T, \tilde \xi)$ is then called the  Markov shift over $X$ induced by the pair $(K,\xi)$. 

We say that a Markov system $(K,\xi)$ is \emph{strongly mixing} if 
$$\lim_{n\to \infty} \norm{ K^n\varphi -  \int \varphi\, d\xi }_\infty=0 $$
with uniform convergence over bounded sets of $C^0(X)$.
It is important to observe that if a  Markov system $(K,\xi)$
is  strongly mixing then  the Markov shift $(T, \tilde{\xi})$ is mixing. See~\cite[Proposition 5.1]{DK-book}.  

\section{Random product of matrices}\label{sec3:010622.2}

In this section we describe the base dynamics associated with random i.i.d. products of matrices generated by a probability measure on $\SL_2(\R)$.

In the subsequent sections $\mu$ is a probability measure on $\SL_2(\R)$ with finite support given by $\supp\mu = \{A_1,\ldots, A_{\kappa}\}\subset \SL_2(\R)$. We write
\begin{align*}
    \mu = \sum_{i=1}^{\kappa}\mu_i\, \delta_{A_i},
\end{align*}
where the components $\mu_i = \mu\{A_i\}>0$ form a probability vector $(\mu_1,\ldots,\mu_{\kappa})$.

\begin{remark}
The positivity requirements $\mu_i>0$  avoids discontinuities as in Kifer counter-example. See~\cite{Ki82} or~\cite[Remark 7.5]{BV16}. 
\end{remark}

\subsection{General locally constant cocycles}\label{260122.5}

\paragraph{Locally constant cocycles:} Let $\prodSpace = \{1,\ldots, \kappa\}^{\Z}$ be the space of sequences in the symbols $\{1,\ldots, \kappa\}$, $\prodMeasure = (\mu_1, \ldots, \mu_\kappa)^{\Z}$ be the Bernoulli product measure on $\prodSpace$ and consider $\shift: \prodSpace\to\prodSpace$ the shift map. Note that the system $(\shift, \prodMeasure)$ is ergodic. We say that the triple $(\prodSpace, \shift, \prodMeasure)$ is the base dynamics determined by $\mu$.

It is important to point out that the base dynamics $(\prodSpace, \shift, \prodMeasure)$ does not depend on which $\kappa$-tuple $(A_1,\dots,A_{\kappa})\in (\SL_2(\R))^{\kappa}$ but only on the values $\mu_i = \mu\{A_i\}$.

Consider the map $\lcCocycle: \prodSpace\to \SL_2(\R)$ given by
\begin{align*}
    \lcCocycle(\ldots, \omega_{-1}, \omega_0,\omega_1,\ldots) := A_{\omega_0}.
\end{align*}
Notice that for each sequence $\omega \in \prodSpace$, $\lcCocycle(\omega)$ only depends on the $0$-th coordinate of the sequence $\omega$. Such maps are known in the literature as \emph{locally constant linear cocycles}. This is an agreed abuse of the term  since for the standard topology in $\prodSpace$,  locally constant observables include a broader class of functions. Since the base dynamics is fixed,  $(A_1,\ldots, A_{\kappa})$ determines the Lyapunov exponent $L(\lcCocycle)$ and for that reason some times we write $L(\lcCocycle) = L(A_1,\ldots,A_{\kappa})$ to emphasize this dependence.
This definition of Lyapunov exponent of   $\lcCocycle$ agrees with the one given in the introduction for the distribution law $\mu$, so that $L(\mu) = L(\lcCocycle) = L(A_1\ldots, A_{\kappa})$.

\paragraph{Uniformly hyperbolic cocycles}
The measure $\mu$, or equivalently, the locally constant cocycle $\lcCocycle:\prodSpace\to \SL_2(\R)$ is said to be \emph{uniformly hyperbolic} if there exist $C>0$ and $\gamma>0$ such that for every $n\geq 1$ and $\omega\in \prodSpace$,
\begin{align}\label{eq:020822.1}
    \norm{\lcCocycle^n(\omega)}\geq Ce^{\gamma n} .
\end{align}
 It is known \cite{Viana-book} that this is equivalent to the existence of two $\lcCocycle$-invariant continuous sections $F^u,\, F^s:\prodSpace\to \bP^1$ such that for every $\omega$  $F^u(\omega)\oplus\, F^s(\omega) = \R^2$ and there exist $C>0$ and $\gamma>0$ such that 
\begin{align*}
    \norm{
        \lcCocycle^n(\omega)|_{F^s(\omega)} 
    }\leq Ce^{-\gamma n}
    \quad
    \text{and}
    \quad
    \norm{
        \lcCocycle^{-n}(\omega)|_{F^u(\omega)}
    } \leq Ce^{-\gamma n}.
\end{align*}

\paragraph{Forward and backward stationary measures:} Consider the transition kernels $Q_+:\bP^1\to \Prob(\bP^1)$ and $Q_{-}:\bP^1\to \Prob(\bP^1)$ defined, respectively, by
\begin{align*}
    Q_+(\hat v) := \sum_{i=1}^{\kappa}\mu_i\, \delta_{A_i\, \hat v}
    \quand
    Q_-(\hat v) := \sum_{i=1}^{\kappa}\mu_i\, \delta_{A_i^{-1}\, \hat v}.
\end{align*}

\begin{definition}
A measure $\eta\in\Prob(\bP^1)$ is called forward, resp. backward, stationary for $\mu$\,  if  \, $Q_+^\ast\, \eta=\eta$, resp. $Q_-^\ast\, \eta=\eta$,
i.e., if $\eta$ is  stationary for $Q_+$, resp.  for $Q_-$.
\end{definition}
 
Notice that a backward stationary measure for $\mu$ is a forward stationary measure for the reverse measure $\mu^{-1}:=\sum_{i=1}^{\kappa} \mu_i\, \delta_{A_i^{-1}}$.

\subsection{Irreducible cocycles}\label{sec:genericLCC}
Throughout this section, unless otherwise explicitly said, we assume that the probability measure $\mu = \sum\mu_i\, \delta_{A_i}\in \spaceProb(\SL_2(\R))$ has positive Lyapunov exponent and is quasi-irreducible. 

It follows that the Markov operator $Q_+:C^0(\bP^1)\to C^0(\bP^1)$ defined by
\begin{align*}
    (Q_+\varphi)(\hv) := \sum_{i=1}^{\kappa}\mu_i\, \varphi(A_i\, \hv),
\end{align*}
preserves the space of $\theta$-H\"older continuous functions $C^{\theta}(\bP^1)$, for some $\theta>0$, and   $Q_+|_{C^{\theta}(\bP^1)}:C^{\theta}(\bP^1)\to C^{\theta}(\bP^1)$ is a quasi-compact and simple operator, i.e., it has a simple largest eigenvalue, namely $1$ associated to the constant functions, and all other elements in the spectrum have absolute value strictly less than $1$.
%
\begin{proposition}
    There exist   a unique forward stationary measure $\plusStat$ and a unique backward stationary measure $\minusStat$ for $\mu$.
\end{proposition}

\begin{proof}
See~\cite[Proposition 4.2]{DK-31CBM}.
\end{proof}

The operator $Q_+$ contracts the H\"older seminorm.

\begin{proposition}\label{prop:020622.4}
    There exist  positive constants $0<\theta<1$, $C$ and $c$ such that
   $$ v_\theta(Q_+^n \varphi)\leq C\,e^{-c\, n}\, v_\theta(\varphi)\qquad \forall \, n\in\N .$$
   for every $\varphi\in C^\theta(\bP^1)$.
\end{proposition}
\begin{proof}
See~\cite[Propositions 4.1 and 4.2]{DK-31CBM}.
\end{proof}

Another consequence of the quasi-compactness is that the locally constant linear cocycle $\lcCocycle:\prodSpace\to \SL_2(\R)$, associated with $\mu$ and  defined on the product space $\prodSpace = \{1,\ldots,\kappa\}^{\Z}$, satisfies uniform large deviation estimates of exponential type in a neighborhood of $\lcCocycle$.

\begin{proposition}\label{ULDE}
    There exist constants $\delta>0$, $C>0$, $\tau>0$ and $\varepsilon_0>0$ such that for every $\varepsilon\in (0,\varepsilon_0)$, for all locally constant $\textbf{B}:\prodSpace\to \SL_2(\R)$ with $\norm{\textbf{B} - \lcCocycle}_{\infty}<\delta$, every $\hat v\in \bP^1$ and   $n\in \N$,
    \begin{align*}
        \prodMeasure 
            \left(\left\{
                \omega\in \prodSpace \, \colon\,  \left|
                    \frac{1}{n}\log\norm{\textbf{B}^n(\omega)\,v} - \Le(\textbf{B})
                \right|\geq \varepsilon
            \right\}\right)  \leq C\, e^{-\tau\varepsilon^2n},
    \end{align*}
    and
    \begin{align*}
        \prodMeasure 
            \left(\left\{
                \omega\in \prodSpace \, \colon\,  \left|
                    \frac{1}{n}\log\norm{\textbf{B}^n} - \Le(\textbf{B})
                \right|\geq \varepsilon
            \right\}\right)  \leq C\, e^{-\tau\varepsilon^2n}.
    \end{align*}
\end{proposition}
\begin{proof}
See Theorem 4.1 and its proof in~\cite{DK-31CBM}.
\end{proof}

Let $0<\theta<1$ be the constant in Proposition~\ref{prop:020622.4}.

\begin{proposition}\label{prop:020622.5}
    There exist constants $C>0$ and $c>0$ such that for any interval $I\subset \bP^1$, every $\hat{v}\in \bP^1$ and $n\geq 1$ we have,
    \begin{align*}
        \eta^{\pm}(I/2) - C\frac{e^{-cn}}{|I|^\theta}
        \leq
        \prodMeasure\left(\left[
            \lcCocycle^{\pm n}(\cdot)\, \hat{v}\in I
        \right]\right)
        \leq \eta^{\pm}(2I) + C\frac{e^{-cn}}{|I|^\theta}.
    \end{align*}
\end{proposition}
\begin{proof}
    Set $\varepsilon = |I|/4$ and consider piece-wise linear functions $f^\pm_\varepsilon\in C^{\theta}(\bP^1)$ 
    such that  $v_{\theta}(f^\pm _{\varepsilon}) = \varepsilon^{-\theta}$, \,
    $0\leq f^-_{\varepsilon} \leq \chi_I \leq f^+_{\varepsilon}\leq 1$,
    \, $f^-_\varepsilon=1$ on $I/2$\,
    and \, $f^+_{\varepsilon}=0$ out of $2I$.
    By Proposition \ref{prop:020622.4} we have that
    \begin{align*}
        \left|
            Q^n_+(f^{\pm}_{\varepsilon}) - \int f^{\pm}_{\varepsilon}\, d\plusStat
        \right|
        &\leq Ce^{-cn}v_{\theta}(f_{\varepsilon})
        = Ce^{-cn}\varepsilon^{-\theta}.
    \end{align*}
    Therefore,
    \begin{align*}
        \prodMeasure\left(\left[
            \lcCocycle^n(\cdot)\, \hat{v}\in I
        \right]\right)
        &= Q_+^n(\chi_I) \geq Q_+^n(f^-_{\varepsilon})
        \geq \int f^-_{\varepsilon}\, d\plusStat - Ce^{-cn}\varepsilon^{-\theta}\\
        &\geq \plusStat(I/2) - 4^{\theta}Ce^{-cn}|I|^{-\theta},
    \end{align*}
    and
    \begin{align*}
        \prodMeasure\left(\left[
            \lcCocycle^n(\cdot)\, \hat{v}\in I
        \right]\right)
        &= Q_+^n(\chi_I) \leq Q_+^n(f^+_{\varepsilon})
        \leq \int f^+_{\varepsilon}\, d\plusStat + Ce^{-cn}\varepsilon^{-\theta}\\
        &\leq \plusStat(2I) + 4^{\theta}Ce^{-cn}|I|^{-\theta}.
    \end{align*}
    The argument for $\lcCocycle^{-n}$ is analogous.
\end{proof}

\subsection{Entropy and dimensions}\label{sec:260722.1}
\paragraph{Entropies:}
For a finitely supported measure $\mu\in\Prob(\SL_2(\R))$ the \emph{Shannon's entropy} defined by
\begin{align*}
    H(\mu):= -\sum_{g\in \supp(\mu)}\mu(\{g\})\log\mu(\{g\}) =-\sum_{i=1}^\kappa \mu_i\, \log \mu_i.
\end{align*}
A measurement of how far the semigroup generated by the $\supp(\mu)$ is from being free is given by
\begin{align*}
    h_{\text{WR}}(\mu) := \lim_{n\to\infty}\frac{1}{n}H(\mu^{*n})
    = \inf_{n\in \N}\frac{1}{n}H(\mu^{*n}),
\end{align*}
which is usually called the \emph{random walk entropy} of $\mu$. It holds that $h_{\text{RW}}(\mu)\leq H(\mu)$ and the equality is equivalent to the semigroup generated by $\supp(\mu)$ being free. This is the typical case.

The \emph{Furstenberg's entropy}, also known as \emph{Boundary} entropy, is defined by
\begin{align*}
    h_F(\eta) := \int\int\log\frac{dg_*\eta}{d\eta}(v)
\, d\eta(v)\, d\mu(g).
\end{align*}
We always have that,
\begin{align*}
    h_F(\eta) \leq h_{\text{RW}}(\mu) \leq H(\mu).
\end{align*}
See \cite[Theorem 2.31]{Fur02} for details.

\paragraph{Dimension:}
Let $\eta$ be a probability measure on $\R$ (or $\bP^1$). For any $t\in \R$, the limits
\begin{align*}
    \overline{\dim}(\eta, t)=\limsup_{\delta\to 0}\frac{
        \log\eta([t-\delta, t+\delta])
    }{\log\delta},
    \hspace{.2cm}
    \underline{\dim}(\eta,t)=\liminf_{\delta\to 0}\frac{
        \log\eta([t-\delta,t+\delta])
    }{\log\delta}
\end{align*}
are called, respectively, the \emph{upper local dimension} and \emph{lower local dimension} of $\eta$ at the point $t\in \R$. We say that $\eta$ is \emph{exact dimensional} if there exists a real number $\alpha\geq 0$ such that $\overline{\dim}(\eta, t) = \underline{\dim}(\eta, t) = \alpha$, for $\eta$-a.e. $t\in \R$. In this case, the number $\alpha$ is the \emph{dimension} of the probability measure $\eta$ and is denoted just by $\dim\eta$.

As mentioned in the introduction the stationary measures of an irreducible cocycle with positive Lyapunov exponent are always exact dimensional.

\paragraph{Entropy deviations:}
\newcommand{\bp}{{\bf{p}}}
Consider $\bp_n:\prodSpace\to \R$, $\bp_{n}(\omega):= \prod_{j=0}^{n-1} p_{\omega_j}$, the function $\varphi:\prodSpace\to \R$,$\varphi(\omega):= -\log  p_{\omega_0} $, 
and notice that 
\begin{align*}
    \Int \varphi\, d\, \prodMeasure = H(\mu)\quad \text{ and } \quad
(S_n\varphi)(\omega):=\sum_{j=0}^{n-1} \varphi(\sigma^j\omega) =-\log \bp_n(\omega).
\end{align*}

\begin{proposition}\label{prop:040722.2}
    Assuming $\Le(\mu)>0$ let  $h:=\max_{1\leq j\leq \kappa} - \log p_j>0$. For every  $n\in\N$ and $\beta>0$,
    \begin{align*}
        \prodMeasure\left(\left\{
            \omega\in \Omega\colon\, \left|
                \frac{1}{n}\log \bp_{n}(\omega)  + H(\mu)
            \right|>\beta
        \right\}\right) \leq 2 \, \exp\left( -n\, \frac{ 2\,\beta^2}{ h^2}\right) .
    \end{align*}
\end{proposition}
\begin{proof}
The large deviation set in the statement is
$$ \Delta_{n}:=\left\{ \omega\in\prodSpace\colon  \left|(S_n\varphi)(\omega)-\mathbb{E}(S_n\varphi) \right| > n\, \beta \right\}$$
and by Hoeffding's inequality~\cite[Theorem 2]{Hof1963} 
$$\prodMeasure(\Delta_n)\leq 2\, \exp\left(-\frac{2\,n^2\,\beta^2}{n\, h^2}\right) = 2\, \exp\left(- n\, \frac{2 \,\beta^2}{h^2}\right)  .$$
\end{proof}

\section{Schr\"odinger cocycles}\label{sec4:010622.3}\label{sec:050722.1}

In this section we present some background in the theory of Schr\"odinger cocycles. The advantage in dealing with this family is the intrinsic relation with the spectral theory of (discrete) Schr\"odinger operators which allow us, among other things, to analyze the behaviour of the Lyapunov exponent in terms of properties of the spectrum of these operators.

\subsection{Schr\"odinger operators and cocycles}
Consider the base dynamics $(X, T, \xi)$, where $T:X\to X$ is a homeomorphism on the compact metric space $X$ and $\xi$ is a probability measure on $X$ such that the system $(T, \xi)$ is ergodic. Fix a continuous function $\phi:X\to \R$.

For each $x\in X$, the (discrete) \emph{Schr\"odinger operator} at $x$ is the self-adjoint bounded linear operator $H_x:l^2(\Z)\to l^2(\Z)$
\footnote{
    \quad $l^2(\Z)$ denotes the set of square-summable sequences $(u_n)_{n\in\Z}$.
}
defined, for $u = (u_n)_{n\in \Z}\in l^2(\Z)$ by
\begin{align*}
    (H_x\, u)_n := -u_{n+1} - u_{n-1} + \phi(T^nx)u_n
\end{align*}
or in short notation
\begin{align*}
    H_x\, u := -\Delta u + \phi_x\, u,
\end{align*}
where $\Delta$ is the Laplace operator and $\phi_x$ is the multiplication by $(\phi(T^nx))_{n\in\Z}$.

It is convenient to express the operator $H_x$ as a matrix in the canonical basis $(e_i)_{i\in\Z}$ of $l^2(\Z)$, where $(e_i)_n = \delta_{i,n}$.
\begin{align*}
    H_x = \left(
        \begin{array}{cccccccc}
            \ddots & \vdots & \vdots & \vdots &  & \vdots & \vdots &  \\
            \dots & \phi(T^{-1}x) & -1 & 0 & \dots & 0 & 0 & \dots \\
            \dots & -1 & \phi(x) & -1 & \dots & 0 & 0 & \dots \\
            \dots & 0 & -1 & \phi(Tx) & \dots & 0 & 0 & \dots \\
             & \vdots & \vdots & \vdots & \ddots & \vdots & \vdots & \\
            \dots & 0 & 0 & 0 & \dots & \phi(T^{n-2}x) & -1 & \dots \\
            \dots & 0 & 0 & 0 & \dots & -1 & \phi(T^{n-1}x) & \dots\\
             & \vdots & \vdots & \vdots &  & \vdots & \vdots & \ddots
        \end{array}
    \right)
\end{align*}
A matrix with this structure where all entries outside the three main diagonals vanish is usually called \emph{tridiagonal matrix}.

Assume that there exists a sequence $u = (u_n)_{n\in \Z}$, not necessarily in $l^2(\Z)$, which satisfies the eigenvalue equation for some $E\in \R$, i.e.,
\begin{align}\label{190122.0}
    H_x\, u = E\, u.
\end{align}
Using the definition of $H_x$, equation~\eqref{190122.0} gives us a second order recurrence equation which can be written in matrix form as
\begin{align*}
    \left(
        \begin{array}{cc}
            \phi(T^{n-1}x) - E & -1 \\
            1 & 0
        \end{array}
    \right)\,
    \begin{pmatrix}
        u_{n-1}  \\
        u_{n-2} 
    \end{pmatrix}=
    \begin{pmatrix}
        u_{n} \\
        u_{n-1}
    \end{pmatrix}.
\end{align*}
This implies that
\begin{align}\label{190122.1}
    \left(
        \begin{array}{cc}
            \phi(T^{n-1}x) - E & -1 \\
            1 & 0
        \end{array}
    \right)\cdot\ldots\cdot
    \left(
        \begin{array}{cc}
            \phi(x) & -1 \\
            1 & 0
        \end{array}
    \right)
    \begin{pmatrix}
        u_0 \\
        u_{-1}
    \end{pmatrix}=
    \begin{pmatrix}
        u_n \\
        u_{n-1}
    \end{pmatrix}.
\end{align}
Hence, if we define the family of cocycles $A_E:X\to \SL_2(\R)$ 
\begin{align*}
    A_E(x) := \left(
        \begin{array}{cc}
            \phi(x) - E & -1 \\
            1 & 0
        \end{array}
    \right),
\end{align*}
then equation \eqref{190122.1} can be rewritten as
\begin{align*}
    A^n_E(x)\,
    \begin{pmatrix}
        u_0\\
        u_{-1}
    \end{pmatrix} =
    \begin{pmatrix}
        u_n\\
        u_{n-1}
    \end{pmatrix}.
\end{align*}
In other words, any (formal) eigenvector $u = (u_n)$ of the Schr\"odinger operator $H_x$ associated with an  eigenvalue $E$ is completely determined by the orbit of the cocycle $A_E$ starting at $(u_0, u_{-1})\in \R^2$. This is one of the first indications of the close relationship between the action of the cocycle $A_E$ and the properties of the spectrum of $H_x$.

The cocycles $A_E: X\to \SL_2(\R)$ are called \emph{Schr\"odinger cocycles} with potential $\phi:X\to \R$, generated by the dynamical system $(X, T, \xi)$. 

\subsection{Integrated density of states and Thouless formula}
\label{truncated IDS subsection}
For each $n\in \N$ and for each $x\in X$, $H^n_x\in \M_n(\R)$ denotes the \emph{truncated Schr\"odinger operator} defined by
\begin{align*}
    H^n_x = \left(
        \begin{array}{cccccccc}
            \phi(x) & -1 & 0 & \dots & 0 & 0\\
            -1 & \phi(Tx) & -1 & \dots & 0 & 0\\
            0 & -1 & \phi(T^2x) & \dots & 0 & 0\\
            \vdots & \vdots & \vdots & \ddots & \vdots & \vdots\\
            0 & 0 & 0 & \dots & \phi(T^{n-2}x) & -1\\
            0 & 0 & 0 & \dots & -1 & \phi(T^{n-1}x)\\
        \end{array}
    \right).
\end{align*}

For any interval $I\subset \R$ denote by $|\spec(H^n_x)\cap I|$ the number of eigenvalues of $H^n_x$ in $I$ counted with multiplicity. With this notation, set for each $x\in X$
\begin{align*}
    \ids_{n,x}(t) := \frac{1}{n}\left
        |\spec(H^n_x)\cap (-\infty, t]
    \right|.
\end{align*}
So,  by definition $\ids_{n,x}(t)$ is a distribution function of a probability measure supported in the spectrum of $H^n_x$.

It is known~\cite[Subsections 3.2 and 3.3]{Da17} that for each $t\in \R$ the limit
\begin{align*}
    \ids(t) = \lim_{n\to\infty} \ids_{n,x}(t),
\end{align*}
exists and by ergodicity of the base dynamics $(X, T, \xi)$ is constant for $\xi$-a.e. $x\in X$. The function $\ids:\R \to [0,\infty)$ is called the \emph{integrated density of states}.

The following equation, known as the
\emph{Thouless formula}, relates the Lyapunov exponent  of a Schr\"odinger cocycle with the integrated density of states.
\begin{align}\label{190122.3}
    L(A_E) = \Int_{-\infty}^{\infty}\log|E - t|\ d\ids(t), \qquad \forall \, E\in\R.
\end{align}
See~\cite[Theorem 3.16]{Da17}. Integrating by parts the Riemann-Stieltjes integral on the right-hand side of equation \eqref{190122.3}, we see that this equation expresses $\Le(A_E)$ as the Hilbert transform of  $\ids(t)$. 
This fact implies, by the work of Goldstein and Schlag, see~\cite[Lemma 10.3]{GS01}, that  the Lyapunov exponent and the integrated density of states must share all `sufficiently nice' modulus of continuity.
These nice moduli of continuity include the H\"older and weak-H\"older regularities. In particular we have:

\begin{proposition}\label{prop:120522.1}
     $\ids(E)$ is not $\beta$-H\"older\, if and only if\, $E\mapsto L(A_E)$ is not $\beta$-H\"older.
\end{proposition}

\subsection{Temple's Lemma}

The Thouless formula allows us to shift the analysis of the regularity from the Lyapunov exponent to  the integrated density of states, and more specifically to the counting of eigenvalues of the truncated Schr\"odinger operators $H^n_x$. An important tool  is  the next linear algebra fact, known as Temple's lemma, which allows us to count eigenvalues by counting instead orthonormal almost eigenvectors.
\begin{lemma}[Temple's lemma]\label{260122.4}
    Let $(V, \langle\cdot,\cdot\rangle)$ be a finite dimensional Hilbert space and let $H:V\to V$ be a self-adjoint linear operator on $V$. Given $\delta>0$ and $\lambda_0\in \R$, assume that there exists a orthonormal set $\{u_1,\ldots, u_k\}\subset V$ such that
    \begin{enumerate}
        \item $\langle Hu_i, u_j\rangle = \langle Hu_i, H u_j\rangle = 0$ \, if \, $i\neq j$,
        \item $\norm{Hu_i - \lambda_0 u_i}\leq \delta$ for every $i$.
    \end{enumerate}
    Then   \, $|\spec(H)\cap (\lambda_0 - \delta, \lambda_0 + \delta)|\geq k$.
\end{lemma}
\begin{proof}
 See~\cite[Lemma A.3.2]{SiTa85}.
\end{proof}
We will say that $u\in V\backslash \{0\}$ is a $\delta$-\emph{almost eigenvector} associated with an \emph{almost eigenvalue} $\lambda_0$ if condition 2 above is satisfied.
\section{Embedding cocycles into Schr\"odinger families}\label{250122.1}

Let $\mu$, as in the previous section, be a probability measure on $\SL_2(\R)$ supported in $\{A_1,\ldots, A_{\kappa}\}$.

We use the notation $\schrMatrix(t)\in \SL_2(\R)$ to denote the \emph{Schr\"odinger matrix}
\begin{align*}
    \schrMatrix(t) = \left(
        \begin{array}{cc}
            t & -1 \\
            1 & 0
        \end{array}
    \right)
\end{align*}

The following lemma is the ground basis of the entire section. The fact that we can decompose any given $\SL_2(\R)$ matrix as a product of four Schr\"odinger matrices provides a way to embed  our random cocycle in a Schr\"odinger cocycle over a  Markov shift. 
\begin{lemma}\label{210122.0}
    For every $B\in \SL_2(\R)$, there exists real numbers $t_0, t_1, t_2$ and $t_3$ such that $B = \schrMatrix(t_3)\,\schrMatrix(t_2)\,\schrMatrix(t_1)\,\schrMatrix(t_0)$.
\end{lemma}
\begin{proof}
Consider first the map $\R^3\ni  (t_1,t_2,t_3)\mapsto S(t_3)\, S(t_2)\, S(t_1)\in \SL_2(\R)$. A direct calculation shows that the range of this map is the set $\SL_2(\R)\setminus \mathcal{M}$ where
$$ \mathcal{M}:=\left\{
    \begin{pmatrix} 
    a & \lambda \\ -\lambda^{-1} & 0 
    \end{pmatrix} \, \colon \, \lambda \neq 0,1, \quad \text{and} \quad a\in \R
\right\} . $$
This implies that the range of the map 
 $\R^3\ni  (t_1,t_2,t_3)\mapsto S(0)\,S(t_3)\, S(t_2)\, S(t_1)\in \SL_2(\R)$ is the set
 $\SL_2(\R)\setminus S(0)\,\mathcal{M}$ where 
 $$ S(0)\, \mathcal{M} = \left\{ \begin{pmatrix} 
 \lambda^{-1} & 0 \\ a & \lambda  \end{pmatrix} \, \colon \, \lambda \neq 0, 1 \quad \text{and} \quad a\in \R  \right\} . $$
Another simple calculation shows that if
\begin{align*}
    (t_1,t_2,t_3,t_4)=(1,\, 1 - \lambda^{-1},\, -\lambda,\, \lambda^{-2} - \lambda^{-1} - a\lambda^{-1})
\end{align*}
then 
$$ \schrMatrix(t_3)\,\schrMatrix(t_3)\,\schrMatrix(t_2)\,\schrMatrix(t_1) = \begin{pmatrix}
\lambda^{-1} & 0 \\ a & \lambda
\end{pmatrix} .$$
Hence  every matrix in $\SL_2(\R)$ is a product of four Schr\"odinger matrices.
\end{proof}

\subsection{Construction of the embedding}
For each $i=1,\ldots, \kappa$, by Lemma \ref{210122.0}, there exists $t^i = (t^i_0,\ldots, t^i_3)\in \R^4$ such that
\begin{align}
\label{S-decomposition}
    A_i = \schrMatrix(t^i_3)\, \schrMatrix(t^i_2)\, \schrMatrix(t^i_1)\, \schrMatrix(t^i_0).
\end{align}

Consider the set $\towerSymbols = \{1,\ldots,\kappa\}\times\{0, 1, 2, 3\}$. We define the following transition kernel $\towerKernel:\towerSymbols\to \spaceProb(\towerSymbols)$, for each element $(i,j)\in \towerSymbols$, 
\begin{align*}
    \towerKernel_{(i,j)} := \left\{
        \begin{array}{cc}
            \delta_{(i,j+1)} & \text{if } j\in \{0,1,2\}\\
            \sum_{k=1}^{\kappa}\mu_k\delta_{(k,0)} & \text{ if } j=3,
        \end{array}
    \right.
\end{align*}
where $\mu_k = \mu(A_k)$, for any $k\in \{1,\ldots, \kappa\}$ and $\delta_{(k,l)}$ denotes the Dirac measure supported in $(k,l)$. Note that the measure
\begin{align*}
    \towerStat = \frac{1}{4}\sum_{j=0}^3\sum_{i=1}^{\kappa}\mu_i\delta_{(i,j)}.
\end{align*}
defines a $\towerKernel$-stationary measure on $\towerSymbols$. Let $\tilde{\towerStat}$ be the Kolmogorov extension of $(K, \towerStat)$ on the  product space $\prodTower = \towerSymbols^{\Z}$. This defines the base dynamics $(\prodTower, \shift, \tilde{\towerStat})$, where $\shift: \prodTower\to \prodTower$ is the shift map and $\supp\tilde{\towerStat}$ is the set of $K$-admissible sequences. 

\subsection{Conjugating the embedded and original cocycle}\label{270122.1}

Consider the real function $\phi:\prodTower\to \R$ defined by
\begin{align*}
    \phi(\zeta) := t^{i_0}_{j_0},
    \quad \text{ where } \; \zeta = ((i_n,j_n))_n,
\end{align*}
where the numbers $t^{i_0}_{j_0}$ were defined in~\eqref{S-decomposition}. We can express the family of Schr\"odinger cocycles, $\lcCocycle_E: \prodTower\to \SL_2(\R)$, with potential $\phi$, generated by the Markov shift $(\prodTower,\, \shift,\, \tilde{\towerStat})$, by
\begin{align*}
    \lcCocycle_E(\zeta) = \schrMatrix(\phi(\zeta) - E),
\end{align*}
for every $E\in \R$ and $\zeta\in \prodTower$. It is important to notice that iterating the cocycle $\lcCocycle_0$ four times we recover the locally constant cocycle $\lcCocycle: \prodSpace\to \SL_2(\R)$. More precisely, for each element $\zeta = ((i_n,j_n))_n\in \prodTower$, with $j_0=0$, consider  the sequence $\omega = (i_{4n})_n\in \prodSpace$. By~\eqref{S-decomposition}  we have that
\begin{align*}
    \lcCocycle_0^4(\zeta)
    &= \lcCocycle_0(\shift^3(\zeta))\, \lcCocycle_0(\shift^2(\zeta))\, \lcCocycle_0(\shift(\zeta))\, \lcCocycle_0(\zeta)\\
    &= S(t^{i_3}_3)\, S(t^{i_3}_2)\, S(t^{i_3}_1)\, S(t^{i_0}_0)\\
    &= A_{i_0} = \lcCocycle(\omega).
\end{align*}
In this case, we say that $\lcCocycle_0:\prodTower\to \SL_2(\R)$ is the \emph{embedding} of the cocycle $\lcCocycle:\Omega\to \SL_2(\R)$ into the Schr\"odinger family $\{\lcCocycle_E:\prodTower\to \SL_2(\R)\}_{E\in \R}$ over $(\prodTower, \shift, \tilde{\towerStat})$.

For each $j\in \{0,1,2,3\}$, set $\prodTower_j := \{(i_n,j_n)_n\in \prodTower;\ j_0 = j\}$. Note that
\begin{align*}
    \prodTower = \bigcup_{j=0}^3\prodTower_j,
\end{align*}
is a partition of the set $\prodTower$ and for each $j\in\{0,1,2,3\}$, $\shift(\prodTower_j) = \prodTower_{j+1\!\!\!\mod\!4}$. In particular, for every $j=0,1,2,3$, $\Sigma_j$ is $\shift^4$-invariant. Denote by $\pi:\prodTower\to \Omega$ the natural projection mapping  $\prodTower\ni (i_n,j_n)_n\mapsto (i_{4n})_n\in \prodSpace$.

Using the notation above we see that $(\prodSpace, \shift, \prodMeasure)$ is a \emph{factor} of  $(\prodTower, \shift, \tilde{\towerStat})$ in the following sense.
\begin{lemma}\label{240122.1}
    The map $\pi:\prodTower\to \prodSpace$ is surjective, \, $\pi_\ast\tilde{\towerStat} = \prodMeasure$ \, and \, $\shift\circ\pi = \pi\circ\shift$. 
    
    \noindent
    Moreover, for each $j\in \{0,1,2,3\}$, $\pi|_{\prodTower_j}$ conjugates $(\prodTower_j, \shift^4, 4\tilde{\towerStat})$ $(\prodSpace, \shift, \prodMeasure)$, where $4\tilde{\towerStat}$ is the normalization of $\tilde{\towerStat}$ on $\prodTower_j$.
\end{lemma}
For the  linear cocycle we have:
\begin{lemma}\label{conjugation1}
For every $j=0,1,2,3$ the linear cocycle
\begin{align*}
    \prodTower_j\times \R^2 \ni (\zeta, v) \mapsto (\shift^{4}(\zeta), \lcCocycle^4_0(\zeta)\, v)\in \prodTower_j\times\R^2
\end{align*}
is conjugated to the linear cocycle
\begin{align*}
    \prodSpace\times\R^2\ni (\omega, v) \mapsto (\shift(\omega), \lcCocycle(\omega)\, v)\in \prodSpace\times\R^2.
\end{align*}
In particular, taking $j=0$ we have that $F^4_{\lcCocycle_0}:\prodTower_0\times\R^2\to \prodTower_0\times\R^2$ is conjugated to $F_{\lcCocycle}:\prodSpace\times\R^2\to\prodSpace\times\R^2$. The same considerations hold for the projectized cocycles.
\end{lemma}
As consequence of the previous lemmas we have
\begin{lemma}\quad 
$\displaystyle 
        \Le(\mu) = \Le(\lcCocycle) = 4\,\Le(\lcCocycle_0).
$
\end{lemma}

Using the conjugation in Lemma \ref{240122.1} we  build the one parameter family of cocycles $\lcCocycle_{(E)}:\prodSpace\to\SL_2(\R)$,
\begin{align*}
    \lcCocycle_{(E)}(\omega) := \lcCocycle^4_E(\zeta),
\end{align*}
where $\zeta = (\pi|_{\prodTower_0})^{-1}(\omega)$. 
The cocycles of this family are locally constant and determined by the probability measures  $\mu_E$ on $\SL_2(\R)$ defined by
\begin{align*}
    \mu_E = \sum_{i=1}^{\kappa}\mu_i\delta_{\lcCocycle_{(E)}(\Bar{i})},
\end{align*}
where  $\Bar{i}$ is any sequence $\omega\in\Omega$ such that $\omega_0=i$, 
for every $i=1,\ldots, \kappa$.
This family is the smooth curve of measures through $\mu$ whose existence is claimed in Theorem~\ref{mainThm}.
It depends analytically on $E$ in the sense that the function $E\mapsto \int \varphi\, d\mu_E=\sum_{i=1}^\kappa \mu_i\, \varphi( \lcCocycle_{(E)}(\Bar i))$ is analytic for every analytic function $\varphi(A)$ on $\SL_2(\R)$. In particular the curve $E\mapsto \mu_{E}$
 is continuous with respect to the weak* topology.
\begin{corollary}
    For every $E\in \R$,\quad $\displaystyle \Le(\mu_E) = \Le(\lcCocycle_{(E)}) = 4\,\Le(\lcCocycle_E)$.
\end{corollary}

\section{Oscillations of the IDS}\label{sec6:010622.5}

Consider the family of Schr\"odinger cocycles $\lcCocycle_E:\prodTower\to \SL_2(\R)$ with potential $\phi:\Sigma\to \R$ over the basis dynamics $(\prodTower, \shift, \tilde{\towerStat})$, as in the Section \ref{250122.1}. 
The purpose of this section is to get a lower bound on the oscillations of the finite scale IDS $\ids_{n,\zeta}$ in terms of counting certain configurations along the orbit of $\zeta$, referred to as $\delta$-matchings.

Let $\{e_1, e_2\}$ be the canonical basis of $\R^2$.
Given $\delta>0$ and  $k\in\N$, we say that  $\zeta\in\Sigma$  has a $\delta$-\emph{matching} of size $k$ at $E$, or a $(\delta,k,E)$-matching, if  
\begin{align*}
    \lcCocycle_E^k(\zeta)\, \hat e_1=\hat e_2  \quad \text{ and } \quad\tau_k(\zeta, E) := \frac{
            \norm{\lcCocycle_E^k(\zeta)\, e_1}
        }{
            \underset{0\leq j\leq k-1}{\max} 
                \norm{
                    \lcCocycle^j_E(\zeta)\, e_1
                }
        } <\delta.
\end{align*}

For each $\zeta\in \Sigma$ and $k\in \N$,   we consider the truncated Schr\"odinger operator $H^k_{\zeta}:\R^k\to \R^k$ defined in Section~\ref{truncated IDS subsection}, which can be described, for $u\in \R^k$, by
\begin{align*}
    H^k_{\zeta}\, u := \left(
        (H^k_{\zeta}u)_0,\ldots, (H^k_{\zeta}u)_{k-1}
    \right),
\end{align*}
where
\begin{align*}
    (H^k_{\zeta}\, u)_j := \left\{
        \begin{array}{ll}
            -u_1 + \phi(\zeta)\, u_0, & \text{if } j = 0  \\
            -u_{j+1} - u_{j-1} + \phi(\shift^j(\zeta))\, u_j, & \text{if } j\neq 0, k-1 \\
            -u_{k-2} + \phi(\shift^{k-1}(\zeta))\, u_{k-1}, & \text{if } j= k-1.
        \end{array}
    \right.
\end{align*}

Let $E\in \R$ and $(v_0,v_{-1})\in \R^2$. Define the sequence $(v_j)_{j\in \Z}$ by the following equation
\begin{align*}
    \lcCocycle_E(\shift^j(\zeta))\,
    \begin{pmatrix}
        v_j\\
        v_{j-1}
    \end{pmatrix}=
    \begin{pmatrix}
        v_{j+1}\\
        v_j
    \end{pmatrix},
\end{align*}
which is equivalent to say that for every $j\in \Z$,
\begin{align}\label{260122.1}
    -v_{j-1} - v_{j+1} + \phi(\shift^j(\zeta))\, v_j = E\,v_j .
\end{align}

Let $\mathbf{e}_0, \ldots, \mathbf{e}_{k-1}$ be the canonical basis of $\R^{k}$.  
\begin{lemma}\label{260122.2}
Given a solution $(v_j)_{j\in\Z}$ of~\eqref{260122.1},
the vector $v^* = (v_0,\ldots, v_{k-1})\in \R^k$  satisfies
    \begin{align*}
        H^k_{\zeta}\, v^* - Ev^* = v_{-1}\mathbf{e}_0 + v_k\mathbf{e}_{k-1}.
    \end{align*}
    Moreover,   $\lcCocycle^k_E(\zeta)\, \hat e_1 = \hat e_2$\, if and only if\, there exists a solution $(v_j)_{j\in\Z}$ of~\eqref{260122.1} such that  $v^*$ is an eigenvector of $H^k_{\zeta}$  with the eigenvalue $E$.
\end{lemma}
\begin{proof}
    The first statement follows from~\eqref{260122.1}. For the second part observe that $\lcCocycle^k_E(\zeta) \hat e_1 = \hat{e}_2$ if and only if  there is a solution of~\eqref{260122.1} such that $v_{-1} = v_k = 0$.
\end{proof}

Consider a large integer  $N=m\, (k+2)$
and split the interval $[0,N-1]$ into $m$ disjoint slots of length $k$, namely $S_j:=[j\,(k+2), j\, (k+2)+k-1]$ for $j=0,1,\ldots, m-1$. The integers $j\,(k+2)-1$ and $j\, (k+2)+k+1$ are referred to as the boundary points of the slot $S_j$. Notice that $\cup_{j=0}^{m-1} S_j$ has $m\, k$ elements which exclude the boundary points of the slots.
We say that  $\zeta\in\Sigma$  has a $(\delta, k,E)$-matching in the slot $S_j$ if $\sigma^{j\, (k+2)}(\zeta)$ has a $(\delta, k,E)$-matching. Next lemma says that when the sequence
$\zeta$  has a $(\delta, k,E)$-matching in the slot $S_j$
we can construct a $\delta$-almost eigenvector for $H^{N}_\zeta$ which is supported in that slot $S_j$.
Moreover, because consecutive slots share no  boundary points in common, if $\zeta$ admits several $(\delta, k)$-matchings in different slots then the corresponding $\delta$-almost eigenvectors are pairwise orthogonal.

\begin{lemma}\label{260122.3}
Given $\zeta\in\Sigma$ and  $j_0\in\N$ such that 
$\sigma^{j_0(k+2)}(\zeta)$ has a $(2^{-1/2} \delta,k,E)$-matching consider the  vector $v^\ast=(v_0,\ldots, v_{k-1})\in \R^k$ with components determined by 
$$\begin{pmatrix}v_j \\ v_{j-1} \end{pmatrix} = \lcCocycle_E^j(\sigma^{j_0(k+2)} \zeta)\, \begin{pmatrix}v_0 \\ 0 \end{pmatrix}$$
where $v_0$ is fixed so that $\max_{0\leq j\leq k-1} |v_j|=1$. Then the vector $v_{j_0,k}(\zeta)\in\R^{N}$, with all coordinates zero except those in the slot $S_{j_0}$
which coincide with the respective coordinates of $v^\ast$, satisfies
    \begin{align*}
        \norm{
            H^{N}_{\zeta}\, v_{j_0,k}(\zeta) - E\, v_{j_0, k}(\zeta)
        }
        <  \delta .
    \end{align*}
    In other words, $v_{j_0,k}(\zeta)$ is an $\delta$-almost eigenvector
     of $H^{N}_{\zeta}$ in the sense of Lemma~\ref{260122.4}.
\end{lemma}
\begin{proof}
    For the sake of  simplicity let $j_0=0$ so that 
    $\zeta=\sigma^{j_0 (k+2)}(\zeta)$ is the sequence with a $(\delta/\sqrt{2},k,E)$-matching.
    By definition of $v_{j_0, k}(\zeta)$ and   Lemma \ref{260122.2} we have that
    \begin{align*}
        (H^{N}_{\zeta} -E)\, v_{j_0, k}(\zeta) = -v_{k-1}\mathbf{e}_k.
    \end{align*}
    Therefore, 
    \begin{align*}
        \norm{
            H^{N}_{\zeta}\, v_{j_0,k}(\zeta) - E\, v_{j_0, k}(\zeta)
        }
        &\leq |v_{k-1}| \leq |v_0|\,
            \norm{\lcCocycle_E^k(\zeta)\, e_1} 
            \leq \sqrt{2}\, \tau_k(\zeta,E) < \delta 
    \end{align*}
    because $v_{k-1}$ is one of the components of $v_0\,\lcCocycle^k_E\, e_1 $ and $$\frac{v_0}{\sqrt{2}}\,  \underset{0\leq j\leq k-1}{\max} 
                \norm{
                    \lcCocycle^j_E(\zeta)\, e_1}
            \leq v_0\, \underset{0\leq j\leq k-1}{\max}\left|
                \langle
                    \lcCocycle^j_E(\zeta)\, e_1, e_1
                \rangle
            \right| 
        =\underset{0\leq j\leq k-1}{\max} |v_j| = 1 .
    $$
\end{proof}

From the  point of view of Mathematical Physics, a
$\delta$-matching determines a $\delta$-almost eigenvector of the Schr\"odinger operator.

Dynamically, these configurations correspond to stable-unstable matchings in the following sense: let $k=k_1+k_2$ be some partition of  $k$ such that both factors in the decomposition
$\lcCocycle^{k}_E(\zeta)=\lcCocycle^{k_2}_E(\sigma^{k_1}(\zeta))\, \lcCocycle^{k_1}_E(\zeta)$ are very hyperbolic  with nearly horizontal unstable direction and almost vertical stable one.
If $k_1, k_2$ are large then $\lcCocycle^{k_1}_E(\zeta) \hat e_1$ is a good approximation of the Oseledets unstable direction $E^u(\sigma^{k_1}\zeta)$
at the point $\sigma^{k_1} (\zeta)$, while $\lcCocycle^{-k_2}_E(\zeta) \hat e_2$ is a good approximation of the stable direction $E^s(\sigma^{k_1}(\zeta))$
at the same point. The condition
$\lcCocycle^k(\zeta)\, \hat e_1=\hat e_2$ is equivalent to the matching
$\lcCocycle^{k_1}_E(\zeta) \hat e_1 = \lcCocycle^{-k_2}_E(\zeta) \hat e_2$
between these two approximate stable and unstable directions at the middle point.
This nearly stable-unstable matching also explains why $\tau_k(\zeta,E)$ should be very small.

The oscillation of the non-decreasing function $\ids$ and its finite scale analogue $\ids_{N,\zeta}$  on some interval $I=[\alpha,\beta]$ are denoted by
$$ \Delta_I \ids:= \ids(\beta)-\ids(\alpha), \; \text{ resp. }\;
\Delta_I \ids_{N,\zeta}:= \ids_{N,\zeta}(\beta)-\ids_{N,\zeta}(\alpha). $$
Denote by $\Sigma(\delta, k, I)$ the subset of $\Sigma$ formed by $(\delta, k, E)$-matching sequences $\zeta\in \Sigma$ with $E\in I$.
\begin{lemma}
    For any interval $I\subset\R$ and  $\zeta\in \Sigma$,
    \begin{align*}
        \Delta_{I_\delta}\ids_{N,\zeta}  \geq  
        \frac{1}{N}\Sum_{j=0}^{m-1}\chi_{\Sigma(\delta,k,I)}(\shift^{j(k+2)} \zeta)   
    \end{align*}
where $I_\delta:=I+[-\delta,\delta]$ is the $\delta$-neighborhood of $I$.
\end{lemma}
\begin{proof}
Let $m\in \N$ and set
\begin{align*}
    \mathcal{Z}_{m,k}(\zeta) := \left\{
        0\leq j \leq m-1\, \colon \, \shift^{j(k+2)} \zeta  \in \Sigma(\delta, k,I)
    \right\} .
\end{align*}
    The set of vectors $\displaystyle \{ v_{j,k}(\zeta) \in \R^{N}\, 
    \colon \ j\in \mathcal{Z}_{k,m}(\zeta)\}$  is orthonormal and by Lemma \ref{260122.3} these are $\delta$-almost eigenvectors. By Lemma \ref{260122.4} there is the same amount of eigenvalues of $H^{N}_{\zeta}$ in $I_\delta$ (counted with  multiplicity). Whence, 
    \begin{align*}
        \Sum_{j=0}^{m-1}\chi_{\Sigma(\delta,k,I)}(\shift^{j\,(k+2)} \zeta)
        =  \left|
            \mathcal{Z}_{k,m}(\zeta) 
        \right|  \, \leq  \, |\spec(H^{N}_{\zeta})\cap I_\delta| 
         = N\,  
           \Delta_{I_\delta}  \ids_{N,\zeta} .
    \end{align*}
\end{proof}
Applying Birkhoff's ergodic theorem sending $m\to\infty$ in the previous lemma we have the following corollary.
\begin{corollary}\label{cor:120522.3}
    For any interval $I\subseteq \R$,  
    \begin{align*}
        \Delta_{I_\delta} \ids  \geq \frac{1}{k+2}\tilde{\towerStat}\left(
            \Sigma(\delta,k,I)
        \right).
    \end{align*}
\end{corollary}

\section{Variation with respect to the energy}\label{sec7:010622.6}

This is the main technical section of the work.

\subsection{Trace property}
The main purpose of this subsection is to prove that if $R_0=\lcCocycle^{4 n_0}_0(\zeta_0)$ is elliptic then as we move the parameter $E$ the rotation angle of $R_E$ varies with non-zero speed around $E=0$. This will be a consequence of the following proposition, which is a general fact about Schr\"odinger matrices. Recall that
\begin{align*}
    S(t) = \left(
        \begin{array}{cc}
            t & -1 \\
            1 & 0
        \end{array}
    \right)
\end{align*}
denotes a Schr\"odinger type matrix.
For a vector $x = (x_1,\ldots, x_n)\in \R^n$, write
\begin{align*}
        S^n(x-E) := S(x_n-E)\,\ldots\, S(x_1-E).
\end{align*}

\begin{lemma}
\label{Avila hyperbolic lemma}
If $E\in\C\setminus\R$ then the matrix $S^n(x-E)$
is hyperbolic.
\end{lemma}

\begin{proof}
See~\cite[Lemma 2.4]{Av2011}. For the sake of completeness we provide a proof of this fact.
By induction the entries in the main diagonal of $S^n(x)$ are polynomials in the variables $x_1,\ldots, x_n$ of degrees $n$ and $n-2$, respectively,
whose monomials have degrees with same parity as $n$,
while the entries  on the second diagonal are polynomials  of degrees  $n-1$, whose monomials have degrees with same parity as $n-1$.
It follows that for all $x\in\R^n$,
$$  \tr(S^n(-x))=(-1)^n\, \tr(S^n(x)).$$
In particular $\tr(S^n(x-E))=\pm\, \tr(S^n(E-x))$
and we only need to consider the case  $\Imp E<0$.
In this case the open set $U:=\{z\in\C\colon \Imp z>0\}$ determines an open cone in $\bP(\C^2)\equiv\C\cup\{\infty\}$. The projective action of the matrices $S^n(x-E)$ with $x\in \R^n$ and $E\in U$ sends $\overline{U}$ inside of $U$. In fact, if $n=1$ and $z\in \overline{U}\backslash \{0\}$ (possibly $z=\infty$) then $-z^{-1}\in U$ and since $\Imp E<0$,
$$  \Imp(S(x_1 - E)\cdot z) = \Imp\left(
    x_1 - E -\frac{1}{z}
\right) \geq -\Imp(E)>0, $$
for every $x_1\in \R$. Otherwise if $z=0$, then $S(x_1 - E)\cdot z = \infty$ and the statement follows iterating once again. The existence of this invariant cone implies that $S^n(x-E)$ is hyperbolic. Similarly,
if $\Imp E>0$ we consider the open set
$U^-:=\{z\in\C\colon \Imp z<0\}$ and prove that, under the projective action, $S^n(x-E)$ sends $\overline{U^-}$ inside of $U^-$.
\end{proof}
\begin{proposition}\label{240122.2}

    For any $n\in \N$, if $|\tr(S^n(x))|<2$, then 
    \begin{align*}
        \frac{d}{dE}\tr\left(
            S^n(x - E)
        \right)\biggr\rvert_{E = 0}\neq 0.
    \end{align*}
\end{proposition}
\begin{proof}
    Define the analytic function $\psi:\C\to \C$ given by
    \begin{align*}
        \psi(E) := \tr\left(
            S^n(x - E)
        \right).
    \end{align*}
    Observe that $\psi$ is real in the sense that $\psi(E)\in \R$ for every $E\in \R$. By assumption $|\psi(0)| < 2$ and so there exists a radius $r_0>0$ such that for every $E$ in the disk centered in $0$ and radius $r_0$, $\mathbb{D}_{r_0}(0)$, we have that $|\psi(E)|< 2$.
    
    Assume by contradiction that $\psi'(0) = 0$. By analiticity of $\psi$, we can write
    \begin{align*}
        \psi(E) = \psi(0) + E^k\Psi(E),
    \end{align*}
    in a neighborhood of $0$, where $k\geq 2$ and $\Psi(0)\neq 0$. In particular, there exists $E^\ast\in \D_{r_0}(0)\backslash \R$ such that $\psi(E^\ast)\in \R$. As a consequence, we conclude that if $\lambda$ and $1/\lambda$ are the eigenvalues of $S^n(x-E^\ast)$, then $\lambda + \lambda^{-1}\in \R$. But, that can only happen if either $|\lambda| = 1$ or else $|\lambda|\neq 1$ and $\lambda$ itself is real. The former can not happen since by Lemma~\ref{Avila hyperbolic lemma} the matrix $S^n(x - E^\ast)$ is hyperbolic.
    The latter implies  that
    \begin{align*}
        \left|
            \tr\left(
                S^n(x - E^\ast)
            \right) 
        \right| = \left|
            \lambda + \frac{1}{\lambda}
        \right| > 2 .
    \end{align*}
    This contradicts the fact that $|\psi(E^\ast)|<2$ and proves the result.
\end{proof}

\begin{lemma}\label{260122.6}
    If $\lcCocycle_0^m(\zeta_0)$ is an elliptic element for some $\zeta_0\in \prodTower$, then
    \begin{align*}
        \frac{d}{dE}\tr(A^m_E(\zeta_0))\biggr\rvert_{E=0}\neq 0.
    \end{align*}
\end{lemma}
\begin{proof}
    Direct consequence of Proposition \ref{240122.2}.
\end{proof}

\begin{proposition}
\label{log concave property}
Given a Schr\"odinger cocycle $A_E: X\to \SL_2(\R)$   with continuous potential $\phi:X\to \R$ and generated by the dynamical system $(X, T, \xi)$,   for all $n\in\N$, $\rho\in (-2,\, 2)$ and $x\in X$,
the polynomial $f_{\rho}:\R\to\R$, $f(E):=\tr ( A_E^n(x)) - \rho$, has $n$ distinct real roots.
\end{proposition}
\begin{proof}
This polynomial can not have a real root $E_0$  with multiplicity $\geq 2$ because this would imply that
$f_{\rho}(E_0)=f'_{\rho}(E_0)=0$, contradicting the conclusion of Proposition~\ref{260122.6}.
To see that it can not have complex non-real roots, assume  that there exists $E_0\in\C\setminus\R$ such that $f_{\rho}(E_0)=0$.
By Lemma~\ref{Avila hyperbolic lemma}, the matrix $S^n(x-E_0)$ is hyperbolic. Denoting by $\lambda$ and $\lambda^{-1}$ the eigenvalues of $S^n(x-E_0)$  we have
$$ \rho = \tr(S^n(x-E_0)) = \lambda+  \lambda^{-1}   $$
which implies that $|\lambda|= 1$. Therefore the matrix
$S^n(x-E_0)$ can not be hyperbolic. This contradiction proves that $f_{\rho}(E)$ can not have complex non-real roots.
\end{proof}

\begin{corollary}
\label{Morse property}
In the previous context,
 $f:\R\to\R$, $f(E):=\tr ( A_E^n(x))$,   is a Morse function, with
$f(E)\geq 2$ at local maxima, and
$f(E)\leq -2$ at local minima.
\end{corollary}
\begin{proof}
    Since $f$ has $n$ different real roots, $f'$ has $n-1$ different real roots and for any pair $a < b$ of roots of $f$, there exists a unique $c\in (a,b)$ root of $f'$. Moreover, by Proposition \ref{log concave property}, if $f''(c)<0$, then $f(c)\geq 2$ and similarly, $f''(c)>0$ implies $f(c)\leq -2$.
\end{proof}

\subsection{Density of tangencies}
In this subsection we prove that cocycles with heteroclinic tangencies are dense outside the class of uniformly hyperbolic cocycles.

Let $\lcCocycle^{4\ell_0}_0(\zeta_0)$ be an elliptic matrix and $\ellipticDelta>0$ be such that $R_E = \lcCocycle^{4\ell_0}_E(\zeta_0)$ is elliptic for every $|E|\leq \ellipticDelta$.

\begin{lemma}\label{lem:230522.1}
    There exist $c>0$ such that for every $m\geq 1$, every $E\in [- \delta_0,\delta_0]$ and every $\hat{v}\in \bP^1$ we have
    \begin{align*}
        \left|
            \frac{d}{dE}R^m_E\, \hat{v}
        \right|\geq m\, c. 
    \end{align*}
\end{lemma}
\begin{proof}
Take $E_0\in [-\delta_0,\delta_0]$ and an inner product in $\R^2$
for which $R_{E_0}$ is a rotation. Then by Proposition~\ref{winding derivative formula} and Lemma~\ref{Schrodinger winding}
\begin{align*}
\frac{1}{m}\, \left|
    \frac{d}{dE}  R_E^m\, \hv \biggr\rvert_{E=E_0} 
\right|
 &=  \frac{1}{m}\, \sum_{j=1}^{m}   R_{E_0} \, \bvec_{j-1}  \wedge  \dot R_{E_0}\, \bvec_{j-1}  
\end{align*}
is bounded away from $0$.
Notice that by compactness of $[-\delta_0,\delta_0]$, all norms associated with inner products that turn the matrices $R_E$ into rotations are uniformly equivalent.
\end{proof}

\begin{proposition}
\label{irrational elliptic existence}
Given $\zeta\in  \prodTower$, $E\in [-\delta_0,\delta_0]$ and $\ell\in\N$, if   $\lcCocycle_{E_0}^{\ell}(\zeta)$ is parabolic then there exist $E$ arbitrary close to $E_0$ such that  $\lcCocycle_{E}^{\ell}(\zeta)$
is   irrational elliptic.
\end{proposition}

\begin{proof}
Consider the function $f:\R\to\R$, $f(E):=\tr (A_E^\ell(\zeta))$.
Assume $\lcCocycle_{E_0}^{\ell}(\zeta)$ parabolic, i.e., $f(E_0)=\pm 2$. When $f'(E_0)\neq 0$, all matrices 
$A_E^{\ell}(\zeta)$ are elliptic in a $1$-sided neighborhood of $E_0$.
On the other hand, if $f'(E_0)=0$ by Corollary~\ref{Morse property}  all matrices 
$A_E^\ell(\zeta)$ are elliptic in a $2$-sided neighborhood of $E_0$.
\end{proof}

\begin{figure}[h]
    \centering
    \includegraphics[width=0.5\textwidth]{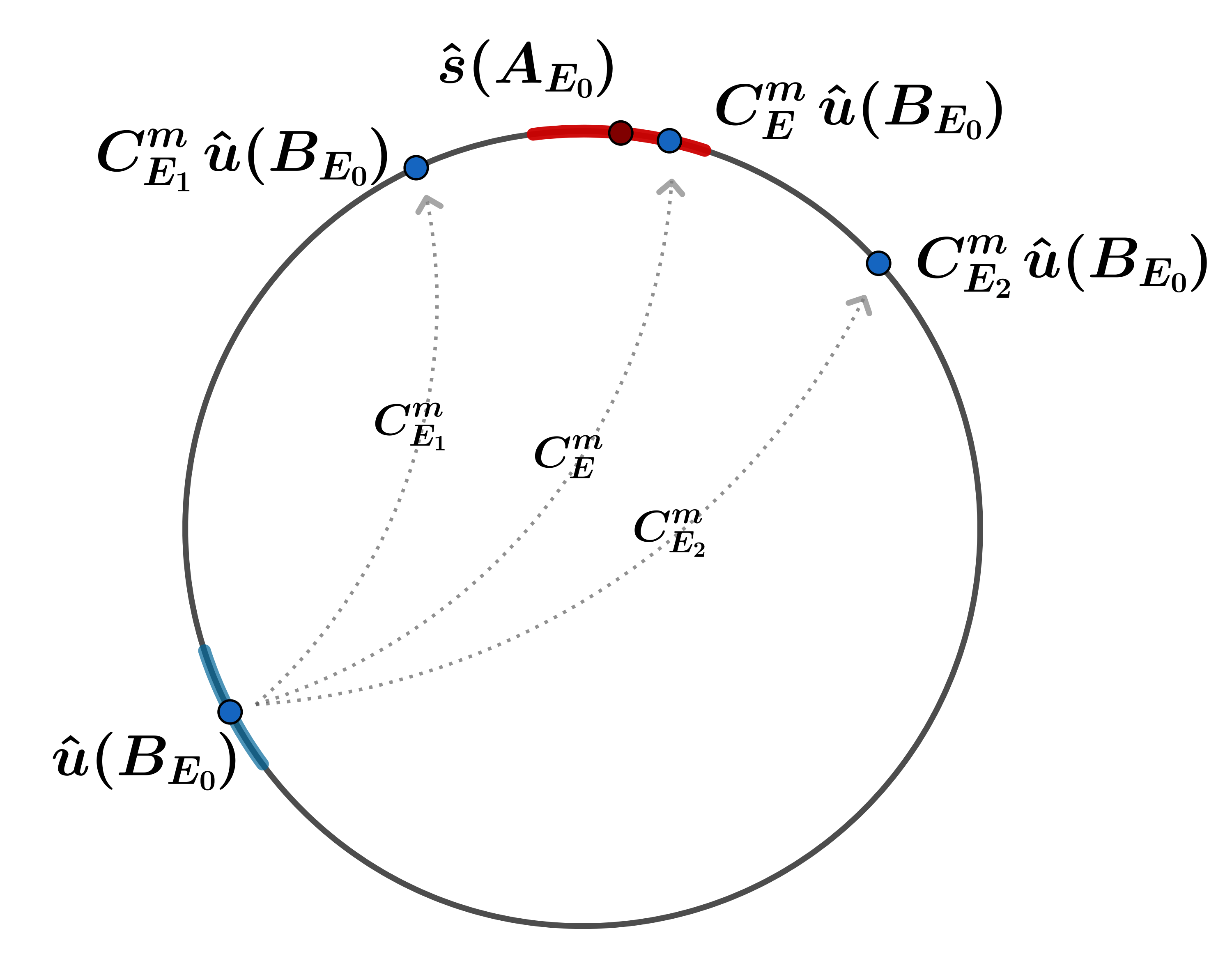}
    \caption{Creation of Tangencies}
    \label{fig:my_label0}
\end{figure}
\begin{proposition}
\label{heteroclinic tangencies  existence}
Assume $\Le(\mu_{E_0})>0$ and $\lcCocycle^{4l}_{E_0}$ is not uniformly hyperbolic, then there exist $E$ arbitrary close to $E_0$ at  which  $\mu_{E}$
admits heteroclinic tangencies.
\end{proposition}
\begin{proof}
By \cite[Thereom 4.1]{ABY10}, either $\mu_{E_0}$ has an heteroclinic tangency, or else the semigroup generated by $\supp\mu_{E_0}$ contains a parabolic or an elliptic matrix. Since $\Le(\mu_{E_0})>0$, $\supp(\mu_{E_0})$ admits hyperbolic matrices $A_{E_0}$ and $B_{E_0}$.
By Proposition~\ref{irrational elliptic existence}
we can assume that $C_{E_0}:=\lcCocycle_{E_0}^{4 \ell}(\zeta)$ is an irrational elliptic rotation, which implies that
the distance $d(C_{E_0}^m\, u(B_{E_0}), s(A_{E_0}))$ gets arbitrary small for some large $m$. On the other hand, the curves $E\mapsto u(B_E), s(A_E)$ are smooth, see Proposition~\ref{speed of stable and unstable directions}, while by Lemma~\ref{lem:230522.1}
the projective curve $E\mapsto C^m_E\, u(B_E)$ has large speed when $m$ is large. Hence the equation 
$C_E^m\, u(B_E)=s(A_E)$ has infinitely many solutions with $E$ arbitrary close to $E_0$.
\end{proof}


\subsection{Projective random walk distribution}
\label{section Dimension of the stationary measures}
In this subsection we establish some estimates on the distribution of the projective random walk, needed to prove Proposition \ref{prop:010722.3}.

\begin{proposition}\label{prop:100722.8}
     Assume that $\Le(\mu) > 0$ and $\mu$ is irreducible. There exist $C>0$ and $t\in (0,1)$ such that
    \begin{align*}
        \sup_{\hat{y}\in \bP^1}\Int_{\bP^1}\frac{1}{d(\hat{x},\, \hat{y})^t}\, d\, \eta(\hat{x}) \leq C.
    \end{align*}
    In particular, $\eta$ is $t$-H\"older, i.e., for every $\hat{x}\in \bP^1$ and $r>0$
    \begin{align*}
        \eta\left(
            B(\hat{x},\, r) 
        \right)\leq Cr^t.
    \end{align*}
\end{proposition}
\begin{proof}
    See \cite{Gu1990} or \cite[Theorem 13.1]{Be2016}.
\end{proof}

The first item of the next proposition corresponds to (13.8)  of Proposition 13.3 in~\cite{Be2016}.

\begin{proposition}\label{P1}
    Assume that $\Le(\mu_{E_0}) > 0$ and $\mu_{E_0}$ irreducible. Given $\beta>0$, there exist constants $C,\, c_1\, c_2 >0$ and $k_0\in \N$ such that for every $\ell, n\in \N$, $n\geq k_0\, \ell$ and directions $\hv, \hw\in \bP^1$, the sets
    \begin{enumerate}[label=\arabic*)]
         \item\label{item:190722.3}
         $\left\{
            \omega\in \prodSpace\colon\,
            \exists\,E,\,  |E - E_0|\leq e^{-c_1n},\,
            \lcCocycle^{ n}_{(E)}(\omega)\, \hv \in B(\hw,\, e^{-\beta \ell})
         \right\}$;
         
         \item\label{item:190722.4}
         $\left\{
            \omega\in \prodSpace\colon\,
            \exists\, E,\, |E- E_0|\leq e^{-c_1n},\,
            \lcCocycle^{- n}_{(E)}(\omega)\, \hw \in B(\hv,\, e^{-\beta \ell})
         \right\}$;
         
         \item\label{item:190722.5}
         $\left\{
            (\omega,\tilde{\omega})\in \prodSpace\times\prodSpace\colon\,
            \exists\,E,\, |E-E_0|\leq e^{-c_1n},\,
            d(\lcCocycle^n_{(E)}(\omega)\, \hv,\,
            \lcCocycle^{-n}_{(E)}(\tilde{\omega})\, \hw) \leq e^{-\beta \ell}
        \right\}$.
     \end{enumerate}
     have probability $\leq Ce^{-c_2\ell}$.
\end{proposition}
\begin{proof}
    By Lemma \ref{lem:020822.1}, there exist constants $C^*,\, c_1^*,\, c_2^* > 0$ such that
    \begin{align}\label{eq:180722.2}
        d\left(
            \lcCocycle^n_{(E_0)}(\omega)\, \hv^+,\,
            \lcCocycle^n_{(E)}(\omega)\, \hv^+
        \right) \leq C^*e^{-c_1^*n},
    \end{align}
    for every $E$ with $|E - E_0|\leq e^{-c_2^*n}$. By Proposition \ref{prop:020622.5} and Proposition \ref{prop:100722.8}, we have that 
    \begin{align*}
        \prodMeasure\left(\left[
            \exists\, E,\, |E-E_0|\leq e^{-c_1n},
        \right.\right.
        &\left.\left.
            \lcCocycle^n_{(E)}(\cdot)\, \hv\in B(\hw,\, e^{-\beta \ell})\,
        \right]\right)\\
        &\lesssim \eta^+\left(
            B(\hw,\, e^{-\beta \ell})
        \right) + C\frac{e^{-cn}}{e^{-c_1\theta \ell}}\\
        &\lesssim e^{-t\beta \ell} + C\frac{e^{-cn}}{e^{-\beta\theta \ell}},
    \end{align*}
    where $n\geq k_0\, \ell$ with $k_0> {\beta(t+\theta)}/{c}$ and $c_2:=t\beta$. The argument to estimate the probability in 2) is entirely  analogous, making use of $\eta^-$ instead of $\eta^+$.

    We now study the probability of the set in \ref{item:190722.5}. We extend the Markov operators $Q_\pm$  to the product space $\bP^1\times\bP^1$ defining a new operator $\textbf{Q}:C^{\theta}(\bP^1\times\bP^1)\to C^{\theta}(\bP^1\times\bP^1)$  by
    \begin{align*}
        (\textbf{Q}\varphi)(\hat{x},\, \hat{y}) := \sum_{i,j = 1}^{\kappa}\mu_i\, \mu_j\, \varphi(A_{i,E_0}\, \hat{x}, A_{j,E_0}^{-1}\, \hat{y}) .
    \end{align*}
    This operator is also a quasi-compact operator and there exists constants $c, C>0$ such that for every observable $\varphi\in C^{\theta}(\bP^1\times\bP^1)$ we have
    \begin{align*}
        v_{\theta}(\textbf{Q}^n\varphi) \lesssim e^{-cn}v_{\theta}(\varphi).
    \end{align*}
    For each $r>0$, let $\Delta_r := \{(\hat x, \hat y)\in\bP^1\times\bP^1 \colon d(\hat x,\hat y)\leq r\}$ and
     $\rho_r:[0,+\infty[\to [0,1]$  be a piece-wise
    linear function supported in $[0,3r]$ such that  $\rho_r(t)= 1$ for $t\in [0,2r]$. Define the $\theta$-H\"older observable $\psi_{r}(\hat x, \hat y):= \rho_r(d(\hat x, \hat y))$,
    with $v_{\theta}(\psi_r) = (2r)^{-\theta}$ and $\chi_{\Delta_r}\leq \psi_r$. Writing $r = 2e^{-c_1n} + e^{-\beta \ell}$, we can use Markov's inequality and Proposition \ref{prop:100722.8} to conclude that
    \begin{align*}
        \prodMeasure\times\prodMeasure &\left(\left\{
            (\omega,\tilde{\omega})\in \prodSpace^2\colon\,
            \exists\,E, |E-E_0|\leq e^{-c_1n},
            d\left(\lcCocycle^n_{(E)}(\omega)\, \hv^+,
            \lcCocycle^{-n}_{(E)}(\tilde{\omega})\, \hv^-
            \right) \leq e^{-\beta \ell}
        \right\}\right)\\
        &\leq \prodMeasure\times\prodMeasure\left(\left\{
            (\omega,\tilde{\omega})\in \prodSpace^2\colon\,
            d\left(
                \lcCocycle^n_{(E_0)}(\omega)\, \hv^+,\,
            \lcCocycle^{-n}_{(E_0)}(\tilde{\omega})\, \hv^-
            \right) \leq 2e^{-c_1n} + e^{-\beta \ell}
        \right\}\right)\\
        & = \textbf{Q}^n(\chi_{\Delta_r})(\hv^+,\, \hv^-)
        \leq \textbf{Q}^n(\psi_{r})(\hv^+,\, \hv^-)\\
        &\leq 
        \left|
            \textbf{Q}^n(\psi_{r}) - \Int_{\bP^1\times\bP^1}\psi_{r}\, d\, (\eta_{E_0}^+\times\eta_{E_0}^-)
        \right|
        +
        \Int_{\bP^1\times\bP^1}\psi_{r}\, d\, (\eta_{E_0}^+\times\eta_{E_0}^-)\\
        &\lesssim e^{-cn}v_{\theta}(\psi_{r}) + (\eta_{E_0}^+\times\eta_{E_0}^-)\left(
            \Delta_{3r}
        \right)\\
        &\lesssim 
        \frac{e^{-cn}}{e^{-\beta\theta \ell}} + 3^t\, (2e^{-c_1n} + e^{-\beta \ell})^t\Int_{\bP^1\times\bP^1}\frac{1}{d(\hat{x},\, \hat{y})^t}\, d\, (\eta_{E_0}^+\times\eta_{E_0}^-)(\hat{x}, \hat{y})\\
        &\lesssim
        \frac{e^{-cn}}{e^{-\beta\theta \ell}} + 3^t\, (2e^{-c_1n} + e^{-\beta \ell})^t\sup_{\hat{y}\in \bP^1}\Int_{\bP^1\times\bP^1}\, \frac{1}{
            d\left(
                \hat{x},\,
                \hat{y}
            \right)^t
        }\, d\, \eta_{E_0}^+(\hat{x})\\
        &
        \lesssim 
        \frac{e^{-cn}}{e^{-\beta\theta \ell}} + (2e^{-c_1 n} + e^{-\beta \ell})^t\;
        \lesssim
        e^{-c_2\ell}.
    \end{align*}
    In the two last inequalities we have used Proposition \ref{prop:100722.8} and that we can increase  $k_0$ so that $k_0> \frac{\beta(t+\theta)}{c}$ and decrease $c_2$ so that  $c_2\leq \beta\theta t$. This completes the proof of the Proposition.
\end{proof}

\subsection{Variation of the `hyperbolic' elements}
In this subsection we establish one of the core proposition for the proof of the Theorem \ref{mainThm}, providing plenty of good hyperbolic words. We will be using the notation introduced in the Section \ref{270122.1}.

Take $\delta_1 = \delta_1(E_0)>0$ as in the Proposition~\ref{ULDE}, in the sense that the large deviations hold uniformly for all cocycles $\lcCocycle_{(E)}$ with $|E - E_0|\leq \delta_1$ and also so that
\begin{align*}
    \lambda:= \min_{|E - E_0|\leq \delta_1} \Le(\mu_{(E)})>0.
\end{align*}

\begin{proposition}
\label{prop:010722.3}
    Assume $\Le(\mu_{E_0})>0$ and $\mu_{E_0}$ irreducible. Given $\beta>0$ there exist constants $\tau>0$  and $N_0\in\N$ such that for every $n\geq N_0$ and every $\hv,\, \hw\in \bP^1$, the set $\mathcal{G}_n(\hv,\hw,\beta, \tau, E_0)$ of all $\omega\in\prodSpace$ satisfying  for all $|E - E_0|\leq e^{-\tau\,n^{1/4}}$:
    \begin{enumerate}
        \item\label{prop:010722.3-item1}
        $\norm{\lcCocycle^n_{(E)}(\omega)\, v} \gtrsim e^{(\lambda - \beta)n}$ and $\norm{\lcCocycle^{-n}_{(E)}(\shift^n\omega)\, w} \gtrsim e^{(\lambda - \beta)n}$;
        
        \item\label{prop:010722.3-item2}
        $\lcCocycle^n_{(E)}(\omega)$ is hyperbolic and $\lambda(\lcCocycle^n_{(E)}(\omega))\gtrsim e^{(\lambda - \beta)n}$;
        
        \item\label{prop:010722.3-item2.5}
        $d(\hv_1^*(\lcCocycle^n_{(E)}(\omega)),\, \hv_2(\lcCocycle^n_{(E)}(\omega))) \gtrsim e^{-\beta n^{1/8}}$.
    \end{enumerate}
    has measure \, $\prodMeasure\left(\mathcal{G}_n (\hv,\hw,\beta,\tau, E_0)
    \right) >1 - \beta$.
\end{proposition}
\begin{proof}
    Split $n$ into blocks of size $n_0\asymp n^{1/4}$. For the sake of simplicity  we assume that $n=m\, n_0$ with $n_0=n^{1/4}$ and $E_0=0$. Consider the sets
    $$\mathcal{B}_{n_0}:= \left\{\omega\in\Omega\colon \exists_{j=0}^{m-1}  \; \norm{\lcCocycle_{(0)}^{n_0}(\sigma^{j\, n_0}\omega) }< e^{  (\lambda-\frac{\beta}{10})\,n_0} \, \vee\, 
    \norm{\lcCocycle_{(0)}^{n_0}(\sigma^{j\, n_0}\omega) }> e^{  (\lambda+\frac{\beta}{10})\,n_0}\right\} $$
    and $\mathcal{B}_{n_0}^\ast:=\mathcal{B}_{n_0}\cup \mathcal{B}_{2 n_0}$, where $\mathcal{B}_{2 n_0}$ is similarly defined. By large deviations, Proposition~\ref{ULDE}, there exists a constant $\tau_1>0$ such that for all large enough $n$, $\prodMeasure(\mathcal{B}_{n_0}^\ast)\leq 2\,n^{3/4}\, e^{-\tau_1\, n^{1/4}}$.
    By finite scale continuity, there exists $\tau>0$ such that
    for all $|E|\leq e^{-\tau\, n^{1/4}}=e^{-\tau n_0}$ and $\omega\notin \mathcal{B}_{n_0}^\ast$,
     $$ e^{  (\lambda-\frac{\beta}{5})\,n_0}\leq \norm{\lcCocycle_{(E)}^{n_0}(\sigma^{j n_0}\omega)} \leq
     e^{  (\lambda+\frac{\beta}{5}) \,n_0}
     \qquad \forall 0\leq j<m  $$
     and 
     $$  e^{  2\, (\lambda-\frac{\beta}{5}) \,n_0}\leq \norm{\lcCocycle_{(E)}^{2 n_0}(\sigma^{j n_0}\omega)} \leq
    e^{2\, (\lambda+\frac{\beta}{5})\, n_0}  
    \qquad \forall 0\leq j<m-1.$$

Consider $(\hv,\, \hw)\in \bP^1\times\bP^1$. We will apply Lemma \ref{190722.10} with the data
\begin{itemize}
    \item $A_j = \lcCocycle^{n_0}_{(E)}(\shift^{jn_0}\omega)$, $j=0,\ldots m-1$;
    \item $\hv = \hv$ and $\hw = \hw$;
    \item $t := \beta n_0^{1/2}$, $\gamma := \frac{4}{5}\beta n_0$ and $\tilde{\lambda} := (\lambda - \frac{\beta}{5})n_0$.
\end{itemize}
Notice that if $\omega\notin \mathcal{B}_{n_0}^*$ the assumptions (a)-(c) of Lemma \ref{190722.10} are automatically satisfied.

Consider $C, c_1, c_2>0$ and $k_0$ given by Proposition \ref{P1} applied with $n = n_0$ and $\ell= n_0^{1/2}$. Denote by $\mathcal{C}_{n_0}(\hv,\, \hw)$ the set of sequences $\omega\in \prodSpace$ such that for every $|E|\leq e^{-\tau n_0}$ ($\tau > c_1$).
\begin{enumerate}[label=(\alph*)]
    \item
    $\min\left\{d\left(
        \lcCocycle^{n_0}_{(E)}(\shift^{(m-1)n_0}\omega)\, \hv,\,
        \hw
    \right),\, d\left(
        \hv,\,
        \lcCocycle^{-n_0}_{(E)}(\shift^{n_0m}\omega)\, \hw
    \right)
    \right\} \geq e^{-\beta n_0^{1/2}}$;
    
    \item
    $\min\left\{d\left(
        \lcCocycle^{n_0}_{(E)}(\omega)\, \hv,\,
        \hw
    \right),\, d\left(
        \hv,\,
        \lcCocycle^{-n_0}_{(E)}(\shift^{n_0}\omega)\, \hw
    \right)
    \right\} \geq e^{-\beta n_0^{1/2}}$;
    
    \item $d\left(
        \lcCocycle^{n_0}_{(E)}(\shift^{(m-1)n_0}\omega)\, \hv,\,
        \lcCocycle^{-n_0}_{(E)}(\shift^{n_0}\omega)\, \hw
    \right) \geq e^{-\beta n_0^{1/2}}$.
\end{enumerate}
If $n_0/\ell = n_0^{1/2}\geq k_0$, then by Proposition \ref{P1} the set $\mathcal{C}_{n_0}^* := \mathcal{C}_{n_0}(\hv,\, \hw)\backslash\, \mathcal{B}_{n_0}^*$ satisfies
\begin{align*}
    \prodMeasure\left(
        \prodSpace\, \backslash\, \mathcal{C}^*_{n_0}
    \right) \leq Ce^{-c_2n_0^{1/2}} = Ce^{-c_2n^{1/8}}.
\end{align*}
If $\omega\in \mathcal{C}^*_{n_0}$ the above conditions (a)-(c) ensure that the hypothesis (d)-(f) of Lemma \ref{190722.10} holds. Therefore items \ref{prop:010722.3-item1}, \ref{prop:010722.3-item2}, and \ref{prop:010722.3-item2.5} are direct consequence of Lemma \ref{190722.10}.
 This concludes the proof of the proposition.
\end{proof}

\begin{proposition}
\label{irred lemma}
If the cocycle $\lcCocycle_{(E_0)}$ is not irreducible with  $\Le(\lcCocycle_{(E_0)})>0$ then there exists $\delta>0$ such that for all $0<|E-E_0|\leq \delta$, the cocycle $\lcCocycle_{(E)}$ is irreducible.
\end{proposition}
\begin{proof}
The cocycle  $\lcCocycle_{(E_0)}$ has either one or two invariant lines, i.e., invariant under all matrices in the support of $\mu_{E}$.
Since  $\lcCocycle_{(E_0)}$ is not uniformly hyperbolic
there exist hyperbolic periodic points $\omega_1$ and $\omega_2$, with periods $n_1$ and $n_2$, respectively, such that
$\hat u(\lcCocycle_{(E_0)}^{n_1}(\omega_1))=\hat s(\lcCocycle_{(E_0)}^{n_2}(\omega_2))$ or/and
$\hat s(\lcCocycle_{(E_0)}^{n_1}(\omega_1))=\hat u(\lcCocycle_{(E_0)}^{n_2}(\omega_2))$, for otherwise a simple argument implies that the reducible cocycle is uniformly hyperbolic, see inequality \eqref{eq:020822.1}.
By Proposition~\ref{speed of stable and unstable directions} the directions
$\hat u(\lcCocycle_{(E)}^{n_i}(\omega_i))$
and
$\hat s(\lcCocycle_{(E)}^{n_i}(\omega_i))$
move in opposite directions with the parameter $E$.
Hence in any case, for $E\neq E_0$ close enough to $E_0$,
together the two matrices $\lcCocycle_{(E)}^{n_i}(\omega_i)$, $i=1,2$, have four distinct invariant directions.
This implies that the cocycle $\lcCocycle_{(E)}$ is irreducible.
\end{proof}

\subsection{Variation of the heteroclinic tangencies}
In this subsection we establish the core proposition \ref{prop:solvingEquations} and \ref{prop:010722.1} for the proof of Theorem \ref{mainThm} which allows us drive matchings and typical tangencies from an existing tangency. 

Consider a family of cocycles $\lcCocycle_{(E)}:\prodSpace\to\SL_2(\R)$
as introduced in Section \ref{270122.1}.
This family has a heteroclinic tangency at $E$ if and only if there exist periodic orbits $\omega_0,\omega_1\in\prodSpace$ with periods $\ell_0, \ell_1\geq 1$ such that $\lcCocycle_{(E)}^{\ell_0}(\omega_0)$ and $\lcCocycle_{(E)}^{\ell_1}(\omega_1)$ are hyperbolic matrices,
and there exists a heteroclinic orbit 
$\omega\in W^u_{\mathrm{loc}}(\omega_0)\cap \sigma^{-k} W^s_{\mathrm{loc}}(\omega_1)$ such that
$$  \lcCocycle_{(E)}^k(\omega)\, \hat u(\lcCocycle_{(E)}^{\ell_0}(\omega_0)) = \hat s(\lcCocycle_{(E)}^{\ell_1}(\omega_1)) . $$
In this case we say that $(B_E,\, C_E,\, A_E)$ is a tangency for $A_{(E)}$ where  $A_{E}=\lcCocycle_{(E)}^{\ell_1}(\omega)$,
$B_{E}=\lcCocycle_{(E)}^{\ell_0}(\omega_0)$
and $C_{E}=\lcCocycle_{(E)}^{k}(\omega) $ are respectively the \emph{target}, the \emph{source} and the \emph{transition} matrix of this heteroclinic tangency. The size of the tangency is by definition the size of the full word $B_E\, C_E\, A_E$ determined by the tangency.

Before entering in the main technical results of this section we state a version of Lemma \ref{Schrodinger winding} suitable for our purposes. We identify the derivative of projective curves such as $E\mapsto \lcCocycle^n_{(E)}(\omega)\, \hv$ with its scalar scalar value. 
\begin{lemma}\label{lem:windingProperty}
    There exists $c_* >0$ such that for all $n\geq 2$, $\omega\in \prodSpace$, $E\in \R$ and $\hv\in \bP^1$, we have
    \begin{align*}
        \frac{d}{dE}\lcCocycle^{-n}_{(E)}(\omega)\, \hv
        < -c_* < 0 < c_* 
        <\frac{d}{dE}\lcCocycle^n_{(E)}(\omega)\, \hv.
    \end{align*}
\end{lemma}
\begin{proof}
    Recall that for each $\omega\in \prodSpace$, $\lcCocycle^n_{(E)}(\omega) = \lcCocycle^n_E(\zeta)$ for some $\zeta\in \prodTower$ and that the cocycle $\lcCocycle_E:\prodTower\to \SL_2(\R)$ is a Schrodinger cocycle. Therefore, this lemma is a direct consequence of Lemma \ref{Schrodinger winding}.
\end{proof}

\begin{definition}\label{def:030822.1}
    Given $\gamma,\, t,\, \rho>0$, we say that a tangency $(B_{E_0},\, C_{E_0},\, A_{E_0})$ for a cocycle $\lcCocycle_{(E_0)}$ is $(\gamma,\, \rho,\, t)$-\emph{controlled} if the following conditions are satisfied:
    \begin{enumerate}
        \item $\min\{\lambda(A_{E_0}),\, \lambda(B_{E_0})\} \geq e^{\gamma}$;
        
        \item $\norm{C_{E_0}}\leq e^{\rho}$;
        
        \item $\min\{d(\hv_1^*(B_{E_0}),\, \hv_2(B_{E_0})),\, d(\hv_1^*(B_{E_0}),\, \hv_2(B_{E_0}))\} \geq e^{-t}$.
    \end{enumerate}
\end{definition}

\begin{figure}[h]
    \centering
    \includegraphics[width=\textwidth]{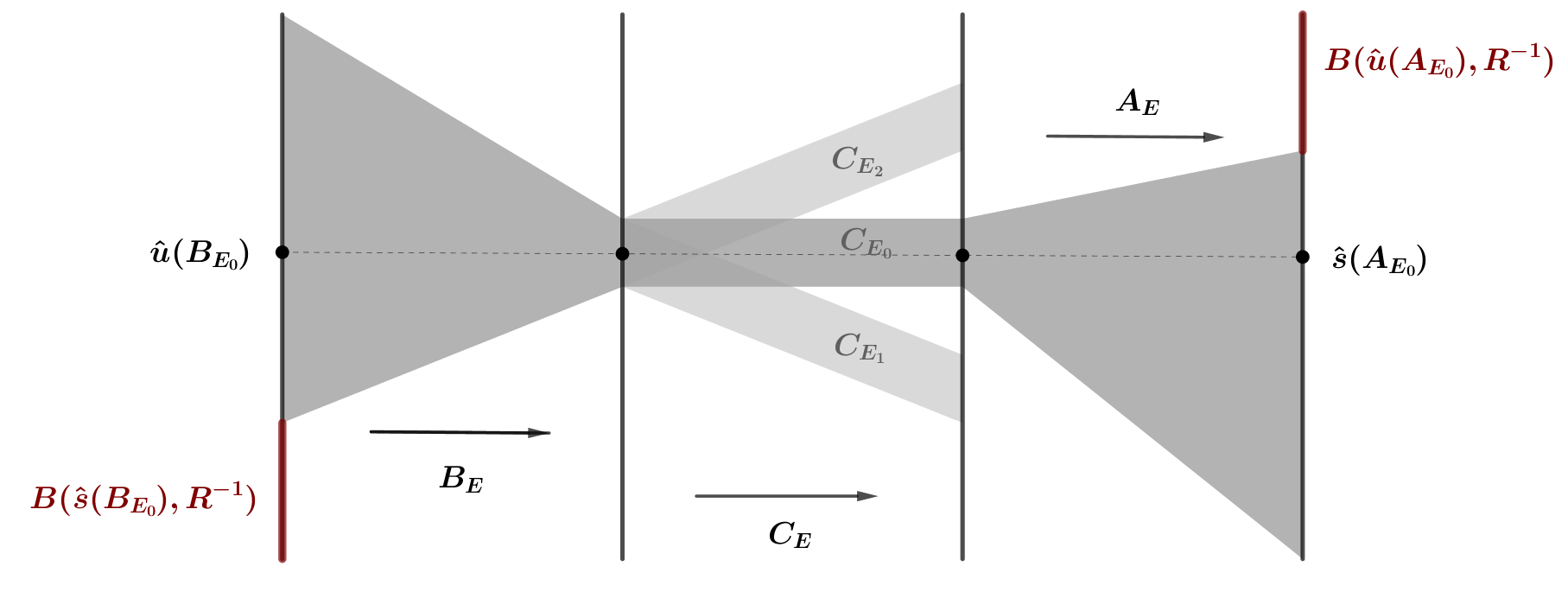}
    \caption{Unfolding the heteroclinic tangency. Vertical lines represent $\bP^1
    $.}
    \label{fig:my_label}
\end{figure}

\begin{proposition}\label{prop:solvingEquations}
    There exists $c_*>0$ such that for every $\beta>0$ and $R>0$ we can find  $\gamma_0$ with the following property: for every $\gamma\geq \gamma_0$, if $(B_{E_0},\, C_{E_0},\, A_{E_0})$ is a $(\gamma,\, \gamma^{1/2},\, \gamma^{1/7})$-controlled tangency for $\lcCocycle_{E_0}$, then defining
    \begin{align*}
        I:=[
            E_0 - 2c_*^{-1}(1+\beta)\, R\, e^{-2\gamma\,  (1-\beta)},\, E_0 + 2c_*^{-1}(1+\beta)\, R\, e^{-2\gamma\, (1-\beta)}
        ],
    \end{align*}
    for every pair of smooth curves $\hv^+,\, \hv^-:I\to \bP^1$ satisfying
    \begin{enumerate}
        \item[A1.] $\hv^+(E_0)\notin \Ball(\hat{s}(B_0),\, R^{-1})$ and $\hv^-(E_0)\notin \Ball(\hat{u}(A_0),\, R^{-1})$;
        
        \item[A2.] $\frac{d}{dE}\hv^+(E)\geq 0$ and $\frac{d}{dE}\hv^-(E)\leq 0$, for every $E\in I$;
    \end{enumerate}
    the equation
    \begin{align*}
        A_E\, C_E\, B_E\, \hv^+(E) = \hv^-(E),
    \end{align*}
    has at least one solution $E_*\in I$.
\end{proposition}
\begin{proof}
    We assume for the sake of simplicity $E_0 = 0$. First observe that using triangular inequality, condition \textit{A1}, Proposition \ref{balanced radius} and the given control of the tangency, there exists $K_0>$ such that
    \begin{align*}
        d(\hv^+(0),\, \hv_2(B_0))
        &\geq d(\hv^+(0),\, \hat{s}(B_0)) - d(\hat{s}(B_0),\, \hv_2(B_0))\\
        &\geq R^{-1} - \frac{K_0}{d(\hv_1^*(B_0),\, \hv_2(B_0))\norm{B_0}^2}\\
        &\geq R^{-1}\left(
            1 - K_0R\, e^{-2\gamma(1 - \frac{1}{2\gamma^{6/7}})}
        \right)
    \end{align*}
    and
    \begin{align*}
        d(\hv^-(0),\, \hv_1^*(A_0))
        &\geq d(\hv^-(0),\, \hat{s}(B_0)) - d(\hat{s}(B_0),\, \hv_2(B_0))\\
        &\geq R^{-1} - \frac{K_0}{d(\hv_1^*(A_0),\, \hv_2(B_0))\norm{A_0}^2}\\
        &\geq R^{-1}\left(
            1 - K_0R\, e^{-2\gamma(1 - \frac{1}{2\gamma^{6/7}})}
        \right).
    \end{align*}
    Now using the previous inequalities jointly with item (b) of Lemma \ref{LA-1} and the control of the transition matrix,
    \begin{align*}
        d(C_0\, B_0\, \hv^+(0),\, C_0\, \hat{u}(B_0))
        &\leq \norm{C_0}^2d(B_0\, \hv^+(0),\, \hat{u}(B_0))\\
        &\leq \norm{C_0}^2\frac{1}{d(\hv^+(0),\, \hv_2(B_0))\norm{B_0}^2}\\
        &\leq R\, \left(
            1- K_0R\, e^{-2\gamma(1 - \frac{1}{2\gamma^{6/7}})}
        \right)^{-1}\, e^{-2\gamma(1 - \gamma^{-1/2})}
    \end{align*}
    and
    \begin{align*}
        d(A_0^{-1}\, \hv^-(0),\, \hat{s}(A_0))
        &\leq \frac{1}{d(\hv^-(0),\, \hv_1^*(A_0))\norm{A_0}^2}\\
        &\leq R\, \left(
            1- K_0R\, e^{-2\gamma(1 - \frac{1}{2\gamma^{6/7}})}
        \right)^{-1}\, e^{-2\gamma}.
    \end{align*}
    Taking  
    \begin{align*}
        \gamma_0 := \max\left\{
            \beta^{-2},\,
            (2\beta)^{-7/6},\,
            \frac{1}{2(1-\beta)}\log\left(
                \frac{K_0R(1+\beta)}{\beta}
            \right)
        \right\},
    \end{align*}
    we conclude that for every $\gamma \geq \gamma_0$
    \begin{align}\label{eq:310722.1}
        d(C_0\, B_0\, \hv^+(0),\, A_0^{-1}\, \hv^-(0)) \leq 2(1+\beta)Re^{-2\gamma(1-\beta)}.
    \end{align}
    
    Choose appropriate projective coordinates in such way to preserve the natural orientation. Consider the functions $f_+,\, f_-: I\to \bP^1$ given by
    \begin{align*}
        f_+(E) = C_E\, B_E\, \hv^+(E)
        \quad
        \text{and}
        \quad
        f_-(E) = A^{-1}_E\, \hv^-(E).
    \end{align*}
    By condition \textit{A2} and Lemma \ref{lem:windingProperty} we have that there exists $c_*>0$ such that $ f'_-(E) < -c_* < 0 < c_* < f'_+(E)$ for every $E\in I$. Moreover, by inequality \eqref{eq:310722.1}, $d(f_+(0),\, f_-(0)) \leq 2(1+\beta)Re^{-2\gamma(1-\beta)}$. Therefore, there exists $|E_*| \leq 2(1+\beta)c_*^{-1}Re^{-2\gamma(1-\beta)}$ such that $f_+(E_*) = f_-(E_*)$, i.e.,
    \begin{align*}
        A_{E_*}\, C_{E_*}\, B_{E_*}\, \hv^+(E_*) = \hv^-(E_*).
    \end{align*}
\end{proof}

We finish this section showing that if the cocycle $A_{E_0}$ has a tangency we can perturb the parameter to produce plenty of new tangencies which are typical with respect to the Lyapunov exponent and the Shannon entropy in a finite scale. Recall the notation of Section \ref{sec:260722.1}.
\begin{proposition}\label{prop:010722.1}
    Assume  the cocycle $\lcCocycle_{(E_0)}$ has a heteroclinic tangency and is irreducible. Given $\beta>0$, there exist constants $C^*_1, C^*_2, c_1^*, c_2^* >0$, a sequence $(l_k)_k\subset \N$, $l_k\to\infty$, and $k_0\in \N$ such that for every $k\geq k_0$ we can find a set $\mathcal{X}_k(\beta)\subset \prodSpace$ with $\prodMeasure(\mathcal{X}_k(\beta)) \geq C_1^*e^{-c_1^*l_k^{1/3}}$ with the following property: for every $\omega\in \mathcal{X}(\beta)$ there exists $E_k = E_k(\omega)$ with $|E_k - E_0|\leq C_2^*e^{-c_2^*l_k^{1/3}}$ such that $\lcCocycle_{(E_k)}$ has a tangency $(P_{E_k}, T_{E_k}, S_{E_k})$, of size $l_k$, satisfying
    \begin{enumerate}
        \item $(P_{E_k}, T_{E_k}, S_{E_k})$ is $(\gamma_k,\, \gamma_k^{1/2},\, \gamma_k^{1/7})$-controlled with $\gamma_k = \frac{l_k}{2}(\lambda - 3\beta)$;
        
        \item $\displaystyle 
            \bp_{l_k}(\omega):=\prod_{j=0}^{l_k-1} p_{\omega_j} 
         \geq  e^{-(H(\mu)+\beta)\, l_k}.$
    \end{enumerate}
\end{proposition}
\begin{proof}
To lighten notations assume $E_0=0$.
Fix $\beta>0$ and let $(B_{0}, C_{0}, A_{0})$ be a tangency for $\lcCocycle_{(E_0)}$.
    Take  integers $p_k, q_k\geq 1$ such that
$$ \left| \frac{p_k}{q_k}-\frac{\log \lambda(B_0)}{\log\lambda(A_0)} \right| <\frac{1}{q_k^2} , $$
or equivalently
\begin{align}\label{eq:020822.2}
    \lambda(B_0)^{q_k} \, \lambda(A_0)^{-\frac{1}{q_k}} < \lambda(A_0)^{p_k} <   \lambda(B_0)^{q_k} \, \lambda(A_0)^{\frac{1}{q_k}},
\end{align}
and write $\lambda_k:=\lambda(A_0)^{p_k}\sim \lambda(B_0)^{q_k}$.
Consider for each $k\geq 1$ the new tangency $(B_{0}^{q_k},\, C_{0},\, A_{0}^{p_k})$ of size $m_k$, also for $\lcCocycle_{(E_0)}$. We claim that this tangency is $(\gamma,\, \gamma^{1/2},\, \gamma^{1/7})$-controlled with $\gamma := (1-\beta)\log\lambda_k$ and $k$ sufficiently large. Indeed, by inequality \eqref{eq:020822.2},
\begin{align*}
    \min\{\lambda(A_0^{p_k}),\, \lambda(B_0^{q_k})\} \geq e^{\gamma},
\end{align*}
for every $k$ sufficiently large. Furthermore, the upper bound for $\norm{C_0}$ and the lower bound for the distances $d(\hv_1^*(A_0^{p_k}),\, \hv_2(A_0^{p_k}))$ and $d(\hv_1^*(B_0^{q_k}),\, \hv_2(B_0^{q_k}))$ can be taken independently of $k$ and so the conditions of Definition \ref{def:030822.1} are automatic satisfied for every $k$ large.
    
For each $R>0$, consider the projective intervals $J^s := \Ball(\hat{s}(B_{0}),\, R^{-1})$ and $J^u := \Ball(\hat{u}(A_{0}),\, R^{-1})$ as well as
\begin{align*}
    I_k := [E_0 - 2 c_\ast^{-1}(1+\beta) R\lambda_k^{-2(1-2\beta)}, E_0 + 2 c_\ast^{-1}(1+\beta) R\lambda_k^{-2(1-2\beta)}].
\end{align*}
    Denote by $\tau_k$ the finite word of size $m_k$ determined by the tangency, i.e., for every $\omega\in [0;\tau_k]$, $\lcCocycle^{m_k}_{(0)}(\omega) = A_{0}^{p_k}\, C_{0}\, B_{0}^{q_k}$.
     Since $\lcCocycle_{0}$ is strongly irreducible, the forward and backward stationary measures $\eta^+$ and $\eta^-$ are non-atomic. Hence, we can choose $R$ sufficiently large so that
    \begin{align}\label{eq:260722.2}
        \eta_{0}^+(6J^s) \leq 1/4
        \quand
        \eta_{0}^-(6J^u) \leq 1/4.
    \end{align}
    
    Given $\hv, \hw\in \bP^1$, consider the set $\mathcal{G}_n:=\mathcal{G}_n(\hv,\, \hw,\, \beta,\, \tau, 0)$ given by Proposition \ref{prop:010722.3}. For each $n\geq 1$, define
    \begin{align*}
        \mathcal{G}_n^u := \left\{
            \omega\in \mathcal{G}_n\colon\,
            \hat{u}(\lcCocycle^n_{(0)}(\omega))\notin 2J^s
        \right\}
        \quad
        \text{and}
        \quad
        \mathcal{G}_n^s := \left\{
            \omega\in \mathcal{G}_n\colon\,
            \hat{s}(\lcCocycle^n_{(0)}(\omega))\notin 2J^u
        \right\}.
    \end{align*}
    Notice that by item 3. of Proposition \ref{prop:010722.3}
    \begin{align*}
        \left\{
            \omega\in \mathcal{G}_n\colon
            \lcCocycle^n_{(0)}(\omega)\, \hv \notin 3J^s
        \right\}\subset \mathcal{G}_n^u
        \quad
        \text{and}
        \quad
        \left\{
            \omega\in \mathcal{G}_n\colon
            \lcCocycle^{-n}_{(0)}(\omega)\, \hw \notin 3J^u
        \right\}\subset \mathcal{G}_n^s.
    \end{align*}
    Thus, by inequality \eqref{eq:260722.2}, Proposition \ref{prop:020622.5} and Proposition \ref{prop:010722.3} we have that
    \begin{align}\label{eq:260722.7}
        \prodMeasure(\mathcal{G}^u_n) \geq \prodMeasure(\mathcal{G}_n) - \prodMeasure([ \lcCocycle^n_{(0)}(\cdot)\, \hv \in 3J^s])
        \geq 1 - \beta - 2\eta_{0}^+(6J^s) \geq \frac{1}{2} - \beta,
    \end{align}
    and similarly,
    \begin{align}\label{eq:260722.8}
        \prodMeasure(\mathcal{G}^s_n) \geq \prodMeasure(\mathcal{G}_n) - \prodMeasure([ \lcCocycle^{-n}_{(0)}(\cdot)\, \hw \in 3J^u])
        \geq 1 - \beta - 2\eta_{0}^-(6J^u) \geq \frac{1}{2} - \beta.
    \end{align}
    
   We define the set $\mathcal{T}_k$ of tangencies by
    \begin{align*}
        \mathcal{T}_k := \mathcal{G}_{ m_k^3}^u
        \cap [ m_k^3;\, \tau_k]
        \cap \shift^{-d_k}\left(
        \mathcal{G}_{ m_k^3}^s
        \right),
    \end{align*}
    where $d_k :=  m_k^3 + m_k$. Take $\omega\in \mathcal{T}_k$ and define the functions $\hv^+,\, \hv^-:I_k\to \bP^1$
    \begin{align*}
        \hv^+(E) := \hat{u}(\lcCocycle^{m_k^3}_{(E)}(\omega))
        \quad
        \text{and}
        \quad
        \hv^-(E):= \hat{s}(\lcCocycle^{m_k^3}_{(E)}(\shift^{d_k}\omega)).
    \end{align*}
    Notice that by definition of $\mathcal{T}_k$,
    \begin{align*}
        \hv^+(0) \notin 2J^s \supset \Ball(\hat{s}(B_0),\, R^{-1})
        \quad
        \text{and}
        \quad
        \hv^-(0)\notin 2J^u\supset \Ball(\hat{u}(A_0),\, R^{-1}).
    \end{align*}
    Moreover, by the Lemma \ref{speed of stable and unstable directions},
    \begin{align*}
        \frac{d}{dE}\hv^+(E)\geq 0
        \quad
        \text{and}
        \quad
        \frac{d}{dE}\hv^-(E)\leq 0.
    \end{align*}
    Thus, we can apply Proposition \ref{prop:solvingEquations}, to guarantee that there exists $E_k = E_k(\omega)\in I_k$, satisfying
    \begin{align*}
        A_{E_k}^{p_k}\,C_{E_k}\,B_{E_k}^{q_k}\,  \hat{u}(\lcCocycle^{m_k^3}_{(E_k)}(\omega)) = \hat{s}(\lcCocycle^{m_k^3}_{(E_k)}(\shift^{d_k}\omega)).
    \end{align*}
    
    Set $l_k:= 2 m_k^3+m_k$ and consider the set
    \begin{align*}
        \mathcal{F}_k(\beta) := \left\{
            \omega\in \prodSpace\colon\,
            \bp_{l_k}(\omega)  \geq e^{-(H(\mu)+ \beta)\,l_k}
        \right\}.
    \end{align*}
    By Proposition~\ref{prop:040722.2} with $n=l_k$ and $\varepsilon=\beta$,
    \begin{align}\label{eq:260722.9}
        \prodMeasure(\mathcal{F}_k(\beta)) \geq 1 - 2e^{-\frac{4}{h^2}\, l_k\,\beta^2},
    \end{align}
    where $h$ is a positive constant depending only on $\mu$.
    
    For each $\omega\in \mathcal{T}_k \cap \mathcal{F}_{k}(\beta)$, define
    \begin{align*}
        P_k := \lcCocycle^{m_k^3}_{(E_k)}(\omega),
        \quad
        T_k := A_{E_k}^{p_k}\,C_{E_k}\,B_{E_k}^{q_k}
        \quad
        \text{and}
        \quad
        S_k := \lcCocycle^{m_k^3}_{(E_k)}(\shift^{d_k}\omega).
    \end{align*}
    and observe that by Proposition \ref{prop:010722.3}, $P_k$ and $S_k$ are hyperbolic and if $\gamma_k := (\lambda - 3\beta)\frac{l_k}{2}$, then $(P_k,\, T_k,\, S_k)$ is a $(\gamma_k,\, \gamma_k^{1/2},\, \gamma^{1/7})$-controlled tangency for the cocycle $\lcCocycle_{(E_k)}$ of size $l_k$. Moreover,
    \begin{align*}
        \lambda(P_k) \gtrsim e^{(\lambda - \beta)m_k^3} > e^{(\lambda - 2\beta)\frac{l_k}{2}}
        \quad
        \text{and}
        \quad
        \lambda(S_k) \gtrsim e^{(\lambda - \beta) m_k^3 }> e^{(\lambda - 3\beta)\frac{l_k}{2}}.
    \end{align*}
    which proves item 1. 
    Item 2 holds because $\omega\in \mathcal{F}_k(\beta)$.
    
   To finish the proposition notice that by inequalities \eqref{eq:260722.7}, \eqref{eq:260722.8} and \eqref{eq:260722.9}
    \begin{align*}
        \prodMeasure(\mathcal{T}_k\cap \mathcal{F}_{k}(\beta))
        &\geq \prodMeasure(\mathcal{T}_k) - \prodMeasure(\prodSpace\backslash\, \mathcal{F}_k(\beta))\\
        &\geq (1/2 - \beta)^2\, \prodMeasure([0;\, \tau_k]) - 2e^{-\frac{4}{h^2}l_k\beta^2}
        \geq (1/2-\beta)^2\, e^{-c\,m_k}
    \end{align*}
    for some constant $c>0$. Taking $\mathcal{X}_k(\beta) := \mathcal{T}_k\cap \mathcal{F}_{k}(\beta)$ completes the proof.
    \end{proof}
\section{Counting Matchings}\label{sec8:010622.7}

The purpose of this section is to give a lower bound for the $\tilde{\towerStat}$-measure of the set of sequences for which we have a $(\delta, k, I)$-matching for some small interval of energies $I$. Throughout this subsection we assume that the cocycle $\lcCocycle_0$ has a heteroclinic tangency and is irreducible. We keep the notation used in the Proposition \ref{prop:010722.1}.

Recall that for a suitable $\delta_1>0$ we use the notation
\begin{align*}
    \lambda = \min_{|E|\leq \delta_1}\Le(\mu_E) > 0.
\end{align*}

\subsection{Subset of matchings}
Take $\beta>0$ and let $\N'$ be the set of sizes $l\in\N$ of the heteroclinic tangencies of  $\lcCocycle_{E_l}$, $(P_{E_l}, T_{E_l}, S_{E_l})$,  given by Proposition \ref{prop:010722.1}, applied with $E_0 = 0$, where $E_l = E_l(\omega)$ for some $\omega\in \prodSpace$ and $|E_l| \leq C_2^\ast\, e^{-c_2^\ast\, l^{1/3}}$ is such that 
$\lcCocycle^{l}_{(E_l)}(\omega) = S_{E_l}\, T_{E_l}\, P_{E_l}$.
Denote by $\tau_l\in\prodSpace$ the finite word of length $l$  associated with the block $S_{E_l}\, T_{E_l}\, P_{E_l}$, for a given size $l\in\N'$.
By item 3 of Proposition \ref{prop:010722.1} 
\begin{align}\label{eq:260722.4}
    \prodMeasure([0; \tau_l])= \bp_{l}(\omega)
    \geq e^{-(H(\mu)+ \beta)\,l}.
\end{align}

To  apply Proposition \ref{prop:solvingEquations} with the tangency $(P_{E_l},\, T_{E_l},\, S_{E_l})$ we consider the balls $J^s_l$ and $J^u_l$ in $\bP^1$ centered respectively in $\hat{s}(P_{E_l})$ and $\hat{u}(S_{E_l})$ with radius $R^{-1}>0$. 
Consider the interval
\begin{align*}
    I_l := [E_l-C\, e^{-l\,(\lambda-\beta)},\, E_l + C\, e^{-l\,(\lambda-\beta)}]
\end{align*}
provided by this proposition
with  $C:= 2 c_\ast^{-1} (1+\beta)  R$.
Choosing $\gamma := \frac{l}{2}(\lambda - \beta)$ by the said proposition the heteroclinic tangency $(P_{E_l},\, T_{E_l},\, S_{E_l})$  is $(\gamma,\, \gamma^{1/2},\, \gamma^{1/7})$-controlled so that the initial assumptions of  Pro\-po\-sition \ref{prop:solvingEquations} are automatically satisfied.

Fix $\tau>0$  given by Proposition~\ref{prop:010722.3} and 
 take $\mathcal{G}_{l^3} := \mathcal{G}_{l^3}(\hat e_1, \hat e_2,\beta, \tau, 0)$ given this proposition. Define then
\begin{align*}
    \Theta_l^u := \left\{
        \omega\in \mathcal{G}_{l^3} \, \colon\,
        \lcCocycle_{E_l}^{l^3}(\omega)\, \hat{e}_1\notin 2J^s_l
    \right\}
    \quad
    \text{and}
    \quad
    \Theta_l^s := \left\{
        \omega\in \mathcal{G}_{l^3} \, \colon\,
        \lcCocycle_{E_l}^{-l^3}(\omega)\, \hat{e}_2\notin 2J^u_l
    \right\}.
\end{align*}
Notice that taking $R$ sufficiently large and applying Proposition \ref{P1} we have
\begin{align}\label{eq:260722.3}
    \prodMeasure(\Theta_l^u) \geq 1/2 - \beta
    \quad
    \text{and}
    \quad
    \prodMeasure(\Theta_l^s) \geq 1/2 - \beta.
\end{align}

Now, we finally define our subset of matchings as
\begin{align*}
    \mathcal{M}_l
    := \Theta_l^u\cap\, \shift^{-l^3}([0;\tau_l])\cap\, \sigma^{-(2 l^3+l)}(\Theta_l^s) .
\end{align*}

\begin{lemma}\label{lem:120522.2}
    For every $l\in\N'$ and $\omega\in \mathcal{M}_l$ there exists $E_l^\ast\in I_l$ such that 
    \begin{enumerate}
        \item $\lcCocycle_{(E_l^\ast)}^{2 l^3+ l}(\omega)\, \hat e_1 = \hat e_2$;
        \item $e^{(\lambda-\beta)\, l^3} \leq
         \norm{\lcCocycle_{(E_l^\ast)}^{l^3}(\omega)\, e_1}\leq \norm{\lcCocycle_{(E_l^\ast)}^{l^3}(\omega)}\leq e^{(\lambda + \beta)\, l^3}$;
        \item $e^{(\lambda-\beta)\, l^3} \leq \norm{\lcCocycle_{(E_l^\ast)}^{-l^3}(\sigma^{2 l^3+l}\omega)\, e_2}\leq \norm{\lcCocycle_{(E_l^\ast)}^{-l^3}(\sigma^{2 l^3+l} \omega)}\leq e^{(\lambda + \beta)\, l^3}$;
        \item $\norm{\lcCocycle_{(E_l^\ast)}^{2l^3+l}(\omega)\, e_1}\leq e^{3 \, \beta \, l^3}$.
    \end{enumerate}  
    Moreover,
    \; $\displaystyle \prodMeasure\left(
            \mathcal{M}_l
        \right) \geq (1/2 - \beta)^2\, e^{
            -l \,(H(\mu)+ \beta)
        }$.
\end{lemma}
\begin{proof}
Fix $\omega\in \mathcal{M}_l$. By Proposition \ref{prop:solvingEquations} we conclude that the equation
    \begin{align}\label{eq:010722.4}
        \lcCocycle^{2 l^3+l}_E(\omega)\, \hat{e}_1 = \hat{e}_2,
    \end{align}
    has at least a solution $E_l^\ast\in I_l$. 
    In fact, as explained above  the heteroclinic tangency $(P_{E_l},\, T_{E_l},\, S_{E_l})$  is $(\gamma,\, \gamma^{1/2},\, \gamma^{1/7})$-controlled so that the initial assumptions of  Proposition \ref{prop:solvingEquations} are automatically satisfied.
    Next consider the curves
    $\hat v^+(E):= \lcCocycle^{l^3}_E(\omega)\, \hat{e}_1$ and $\hat v^-(E):= \lcCocycle^{-l^3}_E(\sigma^{2 l^3+l}\omega)\, \hat{e}_2$. Assumption A1 holds because
    $\omega\in \Theta_l^u$ and $\sigma^{2 l^3+l}\omega\in \Theta_l^s$.  
    Assumption A2 holds by Lemma~\ref{lem:windingProperty}.

    The lower bounds in items 2 and 3 follow from item 1 of 
    Proposition~\ref{prop:010722.3} and the fact that $\omega\in \mathcal{G}_{l^3}(\hat e_1, \hat e_2,\beta, \tau, 0)$. From items 1 and 3 of the said proposition together with conclusion 2) of Proposition~\ref{balanced radius} we get the  upper bounds in items 2 and 3.
  
    Taking unit vectors
    $w_1\in \lcCocycle_{(E^\ast_l)}^{l^3}(\omega)\, \hat e_1$
    and
     $w_2\in \lcCocycle_{(E^\ast_l)}^{l}(\sigma^{l^3}\omega)\, \hat w_1$, by  2 above,
    \begin{align*}
     \norm{\lcCocycle_{(E_l^\ast)}^{2l^3+l}(\omega)\, e_1}&=
     \norm{\lcCocycle_{(E_l^\ast)}^{l^3}(\sigma^{l^3+l}\omega)\, w_2}\, 
     \norm{\lcCocycle_{(E_l^\ast)}^{l}(\sigma^{l^3} \omega)\, w_1}\,
     \norm{\lcCocycle_{(E_l^\ast)}^{l^3}(\omega)\, e_1}\\ 
     &\leq e^{-(\lambda-\beta)\, l^3}\, e^{C\, l} \, e^{(\lambda+\beta) l^3} \leq e^{2 \, \beta \, l^3 + C\, l} \leq e^{3\beta l^3}    ,
    \end{align*}
    which proves item 4.
    
    To finish, using the inequalities in \eqref{eq:260722.4} and \eqref{eq:260722.3} we have
    \begin{align*}
        \prodMeasure(\mathcal{M}_l) = \prodMeasure(\Theta_l^u)\, \prodMeasure(\Theta_l^s)\, \prodMeasure([0; \tau_l]) \geq (1/2 - \beta)^2\, e^{-l\, (H(\mu) + \beta)}.
    \end{align*}
    This completes the proof of the lemma.
\end{proof}

Now we can give a lower bound for the set of matchings. Recall the notation of Section \ref{sec6:010622.5}.
\begin{corollary}\label{cor:310722.1}
For all large $l\in\N'$, if  $n_l:=4(2l^3 + l)$ then 
\begin{align*}
    \tilde{\towerStat}\left(
        \Sigma(e^{-l^3(\lambda - 4\beta)},\, n_l,\,I_l)
    \right) \geq \frac{1}{4}(1/2 - \beta)^2\, e^{-l\, (H(\mu) + \beta)}.
\end{align*}
\end{corollary}
\begin{proof}
    Let $\pi_0 := \pi|_{\prodTower_0}:\prodTower_0\to \prodSpace$ be the conjugation given by Lemma \ref{240122.1}. We claim that $\pi_0^{-1}(\mathcal{M}_l) \subset \Sigma(e^{-l^3(\lambda - \beta)},\, n_l,\,I_l)$. Indeed, by Lemma \ref{lem:120522.2} if $\pi_0(\zeta)\in \mathcal{M}_l$, there exist $E_l^*\in I_l$ such that
    \begin{align*}
        \lcCocycle_{E_l^*}^{n_l}(\zeta)\, \hat{e}_1
        = \lcCocycle^{2l^3 + l}_{(E_l^*)}(\pi_0(\zeta))\, \hat{e}_1
        = \hat{e}_2
    \end{align*}
    Moreover, 
    \begin{align*}
        \tau_{n_l}(\zeta, E_l^*) &
        \leq \frac{
            \norm{
                \lcCocycle_{E^*}^{n_l}(\zeta)\, e_1
            }
        }{
            \norm{
                \lcCocycle_{E^*}^{4\,l^3}(\zeta)\, e_1
            }
        }
        = \frac{
            \norm{
                \lcCocycle_{E_l^*}^{2 l^3 + l}(\pi_0(\zeta))\, e_1
            }
        }{
            \norm{
                \lcCocycle_{E_l^*}^{l^3}(\pi_0(\zeta))\, e_1
            }
        }
        \leq e^{3\,\beta\, l^3 - (\lambda - \beta)\,l^3}
        =e^{-(\lambda - 4\beta)\,l^3}.
    \end{align*}
    This proves that any $\zeta$ is a $(e^{-l^3\,(\lambda - 4\beta)}, n_l, E^*_l)$-matching for some $E^*\in I_l$.
    
    To finish, since $4\tilde{\towerStat}$ is normalization of $\tilde{\towerStat}$ to $\prodTower_0$,
    \begin{align*}
        \tilde{\towerStat}\left(
            \Sigma(e^{-l^3(\lambda - \beta)},\, n_l,\,I_l)
        \right)
        \geq \tilde{\towerStat}(\pi_0^{-1}(\mathcal{M}_l)) = \frac{1}{4}\prodMeasure(\mathcal{M}_l)
        \geq \frac{1}{4}(1/2 - \beta)^2\, e^{-l\, (H(\mu) + \beta)}.
    \end{align*}
    This completes the proof of the corollary.
\end{proof}
\section{Proof of the results}\label{sec9:010622.8}
We keep the notations of the previous section.

\subsection{Proof of Theorem \ref{mainThm}}

We keep the notation of the previous section. Take $\alpha > \frac{H(\mu)}{\Le(\mu)}$ and  choose $\delta>0$ such that
$\lambda:= \min_{|E|\leq \delta} \Le(\mu_E)$ satisfies
$\lambda\, \alpha-H(\mu)>0$. Then take $0<\beta<\lambda$ small enough so that
$\lambda\, \alpha-H(\mu)>2\,\beta+\alpha\, \beta$, which  implies that 
\begin{align}\label{eq:260722.5}
  -H(\mu)-\beta + \alpha\, (\lambda-\beta)>\beta .
\end{align}

By Proposition \ref{prop:120522.1}, to prove Theorem \ref{mainThm} it is enough to prove that the integrated density of states $\ids$ is not $\alpha$-H\"older continuous. By corollaries \ref{cor:120522.3} and \ref{cor:310722.1}, writing $\delta_l : = e^{-l^3(\lambda - 4\beta)}$,
\begin{align*}
    \Delta_{I_l+ [-\delta_l,\, \delta_l]}\ids
    \geq \frac{1}{n_l}\tilde{\towerStat}\left(
        \Sigma(\delta_l,\, n_l,\, I_l)
    \right)
    \geq \frac{1}{4n_l}(1/2-\beta)^2\, e^{-l_l(H(\mu) + \beta)}.
\end{align*}
Thus by inequality \eqref{eq:260722.5},
\begin{align*}
    \frac{
        \Delta_{I_l+ [-\delta_l,\, \delta_l]}\ids
    }{
        |I_l+ [-\delta_l,\, \delta_l]|^{\alpha}
    }
    \gtrsim e^{
        l\left(
            -H(\mu) - \beta + \alpha(\lambda - \beta)
        \right)}
    \gtrsim e^{\beta\, l}.
\end{align*}
Taking $l\in \N'$, $l\to \infty$ we conclude that $\ids$ can not be $\alpha$-H\"older continuous.\qed

\subsection{Proof of Corollary \ref{Cor:010722.8}}
By \cite[Theorem 4.1]{ABY10}, if $\mu$ is not uniformly hyperbolic, then the semigroup generated by $\supp\mu$ must contain a parabolic or elliptic matrix. In either case, by Proposition \ref{heteroclinic tangencies  existence} and Proposition \ref{irred lemma}, we can approximate $\mu$ by measures with finite support admitting tangencies which are irreducible. The result follows by continuity of the quotient $E\mapsto\frac{H(\mu)}{\Le(\mu_E)}$.\qed

\subsection{Proof of Corollary \ref{cor:070122.9}}
By Johnson's theorem \cite{Jo1986}, if $E_0$ is an energy in the almost sure spectrum of the Schr\"odinger operator, then the associated Schr\"odinger cocycle $A_{E_0}$ is not uniformly hyperbolic. Therefore, we can again apply Propositions \ref{heteroclinic tangencies  existence} and \ref{irred lemma} to find energies close to $E_0$ such that the cocycle $A_{E_0}$ is irreducible and has heteroclinic tangencies The result follows by continuity of the quotient $E\mapsto\frac{H(\mu)}{\Le(\mu_E)}$.\qed
\section{Appendix: some linear algebra facts}\label{appendix:LA}

In this appendix we state and prove a few results about the geometry
of the projective action of a matrix $A\in\SL_2(\R)$ in the Euclidean space $\R^2$. Some of these results are well know. Others like propositions~\ref{balanced radius}, ~\ref{lem:080722.3} and Lemma \ref{190722.10} play a key role in logical architecture of our main results.
For the reader's convenience we also state here a version of the Avalanche Principle for $\SL_2(\R)$ matrices.
 
If $\vfrak=\{v_1, v_2\}$ is a basis of $\R^2$ then the dual basis of $\vfrak$ is the unique basis
$\dual{\vfrak}=\{\dual{v_1}, \dual{v_2}\}$  of $\R^2$ such that $\langle \dual{v_i}, v_j\rangle=\delta_{ij}$,
for $i,j=1,2$.

As usual let $J:=\begin{bmatrix} 0 & -1 \\ 1 & 0 
\end{bmatrix}$ denote the $90^{\text{o}}$ rotation matrix.

\begin{lemma}
\label{dual basis and trace}
For a basis $\vfrak=\{v_1, v_2\}$ of $\R^2$, its dual basis is given by $\dual{\vfrak}=\left\{ -\frac{J\, v_2}{v_1\wedge v_2},
\frac{J\, v_1}{v_1\wedge v_2}\right\}$.
In particular, the trace of $A\in\SL_2(\R)$ is given by
$$ \tr(A)= \frac{1}{v_1\wedge v_2} \, \left( A v_1 \wedge v_2  \, +  \, v_1\wedge A v_2 \right). $$
\end{lemma}

\begin{proof}
 For the first part it is enough to check the following relations,
 where we make extensive use the relation  $\langle x, y\rangle = x \wedge J y$.
\begin{align*}
\left\langle  -\frac{J v_2}{v_1 \wedge v_2}, v_1 \right\rangle &= 
-\frac{ \langle J v_2, v_1 \rangle }{v_1 \wedge v_2} 
=-\frac{   J v_2 \wedge J v_1}{v_1 \wedge v_2} =-\frac{  v_2\wedge v_1}{v_1 \wedge v_2} = 1\\
\left\langle  -\frac{J v_2}{v_1 \wedge v_2}, v_2 \right\rangle &= 
-\frac{ \langle J v_2, v_2 \rangle }{v_1 \wedge v_2} 
=0\\
\left\langle  \frac{J v_1}{v_1 \wedge v_2}, v_1 \right\rangle &= 
-\frac{ \langle J v_1, v_1 \rangle }{v_1 \wedge v_2} 
=0\\
\left\langle  \frac{J v_1}{v_1 \wedge v_2}, v_2 \right\rangle &= 
\frac{ \langle J v_1, v_2 \rangle }{v_1 \wedge v_2} 
= \frac{   J v_1 \wedge J v_2}{v_1 \wedge v_2} =\frac{  v_1 \wedge v_2}{v_1 \wedge v_2} = 1
\end{align*}
Denoting the dual basis of $\vfrak=\{v_1, v_2\}$ by
$\dual{\vfrak}=\{\dual{v_1}, \dual{v_2}\}$,
we have
\begin{align*}
    \tr(A)&= \langle \dual{v_1}, A v_1\rangle +
    \langle \dual{v_2}, A v_2\rangle \\
    &= \left\langle  -\frac{J v_2}{v_1 \wedge v_2}, A v_1 \right\rangle
    +
    \left\langle  \frac{J v_1}{v_1 \wedge v_2}, A v_2 \right\rangle\\
    &=   -\frac{\langle J v_2, A v_1 \rangle}{v_1 \wedge v_2}
    +
     \frac{\langle J v_1, A v_2 \rangle}{v_1 \wedge v_2} 
     =   -\frac{J v_2 \wedge J A v_1 }{v_1 \wedge v_2}
    +
     \frac{ J v_1 \wedge J A v_2 }{v_1 \wedge v_2}\\
 &=   -\frac{v_2 \wedge A v_1 }{v_1 \wedge v_2}
    +
     \frac{ v_1 \wedge A v_2 }{v_1 \wedge v_2} = \frac{A v_1 \wedge v_2 }{v_1 \wedge v_2}
    +
     \frac{ v_1 \wedge A v_2 }{v_1 \wedge v_2}  .
\end{align*}
\end{proof}

From Lemma~\ref{LA-1} until Proposition~\ref{lem:080722.3}, we consider the projective distance, 
$$ d(\hat x, \hat y):= \frac{|x\wedge y|}{\norm{x}\norm{y}} =|\sin \measuredangle(x,y)| $$
take a matrix $A\in\SL_2(\R)$ and
 let $\{v_1, v_2\}$ and $\{v_1^\ast, v_2^\ast\}$ be singular orthonormal basis of $A$ characterized by the relations  
 $A\,v_1=\norm{A} v_1^\ast$, \,  $A\, v_2=\norm{A}^{-1}  v_2^\ast$,
 \, $v_2=J v_1$ \,  and  \, $v_2^\ast = J  v_1^\ast$.

\begin{lemma}
\label{LA-1}
For any $\hat x\in \bP^1$, if $x\in \hat x$ is a unit vector,
\begin{enumerate}
    \item[(a)] $\displaystyle \norm{A x}\geq \norm{A}\, d(\hat x, \hat v_2) $,
    \item[(b)]  $\displaystyle  d(A\hat x, \hat v_1^\ast)\leq \frac{1}{d(\hat x, \hat v_2)\, \norm{A}^2}$.
\end{enumerate}

\end{lemma}

\begin{proof}
Writing $x=\langle x, v_1\rangle\, v_1 + \langle x, v_2\rangle\, v_2$ ,
\begin{align*}
\norm{A x} &= \norm{\langle x, v_1\rangle\,\norm{A}\, v_1^\ast + \langle x, v_2\rangle\,\norm{A}^{-1}\, v_2^\ast }\\
&\geq  |\langle x, v_1\rangle|\,\norm{A}\, \norm{v_1^\ast}
= | x\wedge v_2| \, \norm{A} = d(\hat x, \hat v_2)\, \norm{A}.
\end{align*}
Hence
\begin{align*}
  d(A\hat x, \hat v_1^\ast) &= \frac{|(A x)\wedge v_1^\ast| }{\norm{A x}}
  \leq  \frac{|\langle x, v_2\rangle|\, \norm{A}^{-1}\, |v_2^\ast \wedge v_1^\ast| }{d(\hat x, \hat v_2)\, \norm{A} } \leq 
  \frac{1 }{d(\hat x, \hat v_2) \,\norm{A}^2 } .
\end{align*}
\end{proof}

We denote by $\lambda(A)$ the absolute value of the unstable eigenvalue of $A$.

\begin{proposition}
\label{balanced radius}
 If $a:= | v_1^\ast \wedge v_2|$ then 
 \begin{enumerate}[label=\arabic*)]
 \item\label{item:100722.9} $|\tr(A)| \geq a\,\norm{A}$.  
 \item\label{item:100722.10} If $a \norm{A}>2$ \, then\, $A$ is hyperbolic and \\
 $\lambda(A)\geq \frac{1}{2}\,\left( a\,\norm{A} +\sqrt{ a^2\norm{A}^2-4} \right)\asymp a \norm{A}$ \, \text{ as } \, $\norm{A}\to\infty$.
 \item\label{item:100722.11}  There exists a function $k(A):=1+ O(\frac{1}{a^2 \norm{A}^2})$ such that 
 $$d(\hat s(A), \hat v_2)\leq  \frac{k(A)}{a\,\norm{A}^2} 
 \;\text{  and } \; 
    d(\hat u(A), \hat v_1^\ast)\leq \frac{k(A)}{a\,\norm{A}^2} .$$
Moreover, for any $\hat x\in\bP^1$,
$$ d(\hat u(A), A\,\hat x)\leq
\frac{k(A)}{\norm{A}^2}\, \left( \frac{1}{a} + \frac{1}{ |x\wedge v_2 |} \right) .$$
 \end{enumerate}
\end{proposition}

\begin{proof}
Item (1):  
Consider the basis $\vfrak=\{v_1^\ast, v_2\}$.
By Lemma~\ref{dual basis and trace},
\begin{align*}
| \tr(A) |
&=  \frac{1}{|v_1^\ast\wedge v_2|} \, \left|  A v_1^\ast  \wedge v_2 + v_1^\ast\wedge A v_2 \right| \\
&=  \frac{1}{|v_1^\ast\wedge v_2|} \, \left|
\left( \langle v_1^\ast, v_1\rangle \,A v_1
+ \langle v_1^\ast, v_2\rangle\,    \wedge A v_2 \right)) \wedge v_2 +   \norm{A}^{-1} (v_1^\ast\wedge v_2^\ast) \right| \\
&=  \frac{1}{|v_1^\ast\wedge v_2|} \, \left| 
\langle v_1^\ast, v_1\rangle  (v_1^\ast \wedge v_2 ) \norm{A} +  \left\{ \langle v_1^\ast, v_2\rangle  (v_2^\ast \wedge v_2 )
+   (v_2^\ast\wedge v_1^\ast) \right\} \norm{A}^{-1} \right| \\
&=  \frac{1}{|v_1^\ast\wedge v_2|} \, \left[ 
(v_1^\ast\wedge v_2)^2 \norm{A}  +  \left(1-\langle v_1^\ast, v_2\rangle^2\right) \, \norm{A}^{-1}  \right] \\
&\geq |v_1^\ast\wedge v_2|\, \norm{A} 
+  |v_1^\ast\wedge v_2|^{-1}\left(1-\langle v_1^\ast, v_2\rangle^2\right) \, \norm{A}^{-1} \geq a\, \norm{A}.
\end{align*}

\noindent
Item (2): If $a\, \norm{A}>2$ then
$|\tr(A)|\geq a\,\norm{A}>2$ and $A$ is hyperbolic.
Therefore
$$ a\,\norm{A}\leq |\tr(A)| =\lambda+\lambda^{-1} \; \text{ with }\;  \lambda=\lambda(A)  . $$
Solving in $\lambda$ we get
$$ \lambda \geq \frac{1}{2}\,\left( a\, \norm{A} + \sqrt{a^2\,\norm{A}^2-4  } \right) . $$

\noindent
Item (3): 
Because $v_1^\ast=\langle v_1^\ast, v_1\rangle v_1+ \langle v_1^\ast, v_2\rangle v_2$,
$$A v_1^\ast=\langle v_1^\ast, v_1\rangle \norm{A} v_1^\ast + \langle v_1^\ast, v_2\rangle \norm{A}^{-1} v_2^\ast$$ 
and whence
$$ d(\hat v_1^\ast, A\, \hat v_1^\ast) =\frac{|\langle v_1^\ast, v_2\rangle| \norm{A}^{-1} |v_1^\ast \wedge v_2^\ast|}{\sqrt{ a^2 \norm{A}^2 + (1-a^2) \norm{A}^{-2}}}
\leq  \frac{1}{  a  \norm{A}^2} . $$
Since near $\hat u (A)$ the projective map $\hat A$
is a Lipschitz contraction with Lipschitz constant of order $\lambda(A)^{-2}$,  
$$ d(\hat u(A), \hat v_1^\ast) \leq \frac{ \frac{1}{  a  \norm{A}^2} }{1-O(\lambda(A)^{-2})} =  \frac{1}{  a  \norm{A}^2}
\,\left( 1+ O\left(\frac{1}{a^2\norm{A}^2}\right) \right) .$$
The bound on $d(\hat s(A), \hat w)$ follows from the previous inequality applied to $A^{-1}$.
By Lemma~\ref{LA-1}(b),
$$ d(A\hat x, \hat v_1^\ast)\leq \frac{1}{d(\hat x, \hat v_2)\, \norm{A}^2} . $$
Hence  by the triangle inequality
$$d(\hat u(A), A\,\hat x)\leq
d(\hat u(A), \hat v_1^\ast)+d(\hat v_1^\ast, A\,\hat x)\leq 
\frac{k(A)}{\norm{A}^2}\, \left( \frac{1}{a} + \frac{1}{ |x\wedge v_2 |} \right) .$$
 \end{proof}

\begin{lemma}\label{lem:080722.1}
   If $\norm{w} = 1$ and $d(\hw,\, \hv_1) < \sqrt{ 1 - \norm{A}^{-2} }$, then $\norm{A w} \geq 1$.
\end{lemma}
\begin{proof}
    By Lemma~\ref{LA-1}(a),
    \begin{align*}
        \norm{A w}\geq \norm{A }\, d(\hat w,\hat v_2) = \norm{A }\, \sqrt{ 1 - d(\hat w, \hat v_1)^2 } \geq \norm{A}\,\norm{A}^{-1}=1 .
    \end{align*}
\end{proof}

\begin{lemma}\label{Lemma for LA1}
   If \, $d(A\, \hv, \hv_1^\ast) \leq \norm{A}^{-1}$, \,
   $d(A\, \hv, \hw) \geq a>0$\, and \,  $a\norm{A}> 2$\, then 
   $d(A^{-1}\, \hw, \hv_2) < \norm{A}^{-1}$.
\end{lemma}

\begin{proof}
    Note that $d(\hw,\, \hv_1^\ast)\geq d(A\hv,\, \hw) - d(A\hv,\, \hv_1^\ast)$. By assumption this implies   that $d(\hw,\, \hv_1^\ast) \geq a - \norm{A}^{-1} \geq \norm{A}^{-1}$, or equivalently \,  $d(\hw,\, \hv_2^\ast) \leq \sqrt{1 - \norm{A}^{-2}}$. Applying Lemma~\ref{lem:080722.1} to $A^{-1}$ we conclude that $\norm{A^{-1}w} \geq 1$.
    Therefore, writing
    \begin{align*}
        v &= x_v \, v_1 + y_v \, v_2, \; \text{ with } \; x_v^2+y_v^2=1 ,\\
        w &= x_w \, v_1^\ast + y_w \, v_2^\ast,  \; \text{ with } \; x_w^2+y_w^2=1 ,\\
        A\, v &= x_v\,\norm{A}\, v_1^\ast + y_v\,\norm{A}^{-1}\, v_2^\ast\\
        A^{-1}\, w &= x_w\,\norm{A}^{-1}\, v_1 + y_w\, \norm{A}\, v_2,
    \end{align*}
 we have that
    \begin{align*}
        d(A^{-1}\, \hat w,\, \hat v_2) = \frac{
            |A^{-1}\, w\,
            \wedge  v_2|
        }{\norm{A^{-1}\, w}} = \frac{|x_w|\norm{A}^{-1}}{\norm{A^{-1}w}}
        \leq   \norm{A}^{-1} .
    \end{align*}
\end{proof}

\begin{proposition}\label{lem:080722.3}
 If \,
$d(A\,\hat v, \hat w)\geq a$, \, $d(A^{-1}\,\hat w, \hat v)\geq a$  \, and  \, $a  \norm{A}>2 $ \,
then
$$ \max\left\{ \, d(A\, \hat v, \hat v_1^\ast), \, d(A^{-1}\, \hat w, \hat v_2) \,\right\}  \leq \frac{2}{ a\,\norm{A}^2+\norm{A} \sqrt{a^2\norm{A}^2-4}}\asymp\frac{1}{a \norm{A}^2} .$$
\end{proposition}
 
\begin{proof}
Without loss of generality me may assume that 
$$a\leq d(A \hat v, \hat w)\leq d(A^{-1} \hat w, \hat v),$$
for otherwise we would replace the roles of $\hat v$ and $\hat w$, respectively of $A$ and $A^{-1}$.
To prove the inequalities above we derive a system of recursive inequalities, which by iteration lead  to fixed point bound. For this scheme to work we need the following preliminary inequalities:
\begin{equation}
\label{claim2}
d(A \hat v, \hat v_1^\ast)\leq  \norm{A}^{-1} \quad \text{ and } 
\quad  d(A^{-1} \hat w, \hat v_2) \leq \norm{A}^{-1} .
\end{equation} 

Choose $v$ to have norm $1$ and normalize $w$ so that $A^{-1} w$ has norm $1$. There exist coordinates
$(x_v,y_v)$ and $(x_w, y_w)$ in the unit circle such that
\begin{align*}
  v &= x_v   v_1^\ast + y_v   v_2^\ast , \\
 A^{-1} w &= x_w   v_1  + y_w   v_2  ,\\
A v &= x_ v \norm{A} v_1^\ast + y_v \norm{A}^{-1} v_2^\ast , \\
w &= x_ w \norm{A} v_1^\ast + y_w \norm{A}^{-1} v_2^\ast .
\end{align*} 
We have 
$$ d( A\hat v, \hat v_1^\ast) =\frac{|A v \wedge v_1^\ast|}{\norm{A v}}
=\frac{|y_v|\, \norm{A}^{-1}}{\sqrt{ x_v^2 \norm{A}^2 + y_v^2 \norm{A}^{-2}}} $$
and similarly
$$ \frac{ | x_v y_w - x_w y_v|}{\sqrt{ x_v^2 \norm{A}^2 + y_v^2 \norm{A}^{-2}}}
=d(A \hat v, \hat w) \leq d(A^{-1} \hat w, \hat v) =  | x_v y_w - x_w y_v|  .
$$
This implies that \, 
$x_v^2 \norm{A}^2 + y_v^2 \norm{A}^{-2}\geq 1$ and whence
$$ d( A\hat v, \hat v_1^\ast)\leq
|y_v|\, \norm{A}^{-1}\leq \norm{A}^{-1} .$$
This proves the first inequality in~\eqref{claim2}.
The second follows from Lemma~\ref{Lemma for LA1}.

We establish next the mentioned recursive inequalities.
Since
$$ a\leq d(A^{-1}\hat w, \hat v)\leq
 d(A^{-1}\hat w, \hat v_2)+  d(\hat v, \hat v_2)$$
 we have\, 
$$d(\hat v, \hat v_2)\geq a- d(A^{-1}\hat w, \hat v_2)
\geq a-\norm{A}^{-1}\geq \norm{A}^{-1} > 0  $$
and by Lemma~\ref{LA-1} (b),
\begin{equation}
\label{eq1}
\norm{A}\, d(A\hat v, \hat v_1^\ast)\leq \frac{1}{d(\hat v, \hat v_2)\, \norm{A}}
\leq  \frac{1}{a \norm{A} - d(A^{-1}\hat w, \hat v_2)\, \norm{A}} .
\end{equation}
Similarly,
$$ a\leq d(\hat w, A \hat v)\leq
 d(\hat w, \hat v_1^\ast)+  d(\hat v_1^\ast, A \hat v)$$
implies that, 
$$d(\hat w, \hat v_1^\ast)\geq a- d(\hat v_1^\ast, A \hat v) \geq a-\norm{A}^{-1}\geq \norm{A}^{-1}>0 .$$
Hence, as before,
\begin{equation}
\label{eq2}
\norm{A}\, d(A^{-1}\hat w, \hat v_2)\leq \frac{1}{d(\hat w, \hat v_1^\ast)\, \norm{A}}
\leq  \frac{1}{ a \norm{A} -d(\hat v_1^\ast, A \hat v)\, \norm{A}} .
\end{equation}

To solve the recursive inequalities~\eqref{eq1} and~\eqref{eq2}, consider  the $1$-parameter family of partial maps
$F_b:\R^2\to\R^2$, $F_b(x,y):=\left(\frac{1}{b-y},\frac{1}{b-x} \right)$. For $b >2$, each component of $F_b$ is a well-defined contraction of the interval $[0,1]$.
Hence $F_b$ leaves the square $[0,1]^2$ invariant  and is a strict contraction with unique fixed point
$(x_\ast, y_\ast):= \left(\frac{2}{b+\sqrt{b^2-4}},
\frac{2}{b+\sqrt{b^2-4}}\right)$.
Moreover, the maps $F_b$ preserve the usual partial order of $\R^2$, defined by
$$(x,y)\geq (x',y')\; \text{  if } \; x\geq x'\, \text{ and } \,  y\geq y' .$$
Setting  $b:=a\norm{A}$   and   $(x_0,y_0) := \left( \norm{A} d(A\hat v, \hat v_1^\ast) , \, 
\norm{A} d(A^{-1}\hat w, \hat v_2)  \right)$,~\eqref{eq1} and~\eqref{eq2} are equivalent to \,
$(x_0,y_0) \leq F_{b}(x_0,y_0)$, while~\eqref{claim2}
ensures that $(x_0,y_0)\in [0,1]^2$.
Hence we obtain, inductively,  that $(x_0,y_0) \leq F_{b}^n(x_0,y_0)$
for all $n\geq 1$, and taking the limit as $n\to\infty$, $(x_0,y_0) \leq (x_\ast, y_\ast)$. This concludes the proof.
\end{proof}

\begin{remark}
    By the previous lemma, if $a^\ast\norm{A}>2$ then $A$ is hyperbolic and
    \begin{align*}
        \lambda(A) \geq \frac{a^\ast\norm{A} + \sqrt{ a^\ast2\norm{A}^2 - 4} }{2}.
    \end{align*}
\end{remark}

Given $A\in\SL_2(\R)$ with $\norm{A}>1$, denote by
$\hat v_1(A)$, $\hat v_2(A)$, $\hat v_1^\ast(A)$ and $\hat v_2^\ast(A)$
the unique projective points such that taking unit vectors $v_i\in \hat v_i(A)$ and $v_j^\ast \in \hat v_j(A)$,  with $i,j=1,2$, $\{v_1, v_2\}$ and $\{v_1^\ast, v_2^\ast\}$ are singular  basis of $A$ characterized by the relations  
 $A\,v_1=\norm{A} v_1^\ast$ and  $A\, v_2=\norm{A}^{-1}  v_2^\ast$.

\begin{lemma}
\label{eq:090722.10}
Given $A, A'\in \SL_2(\R)$ with $\norm{A}, \norm{A'}>1$,
    \begin{align*}
    \frac{\norm{A'A}}{\norm{A'}\, \norm{A}} \, \sqrt{1-\frac{ \norm{A}^{-4} + \norm{A'}^{-4}}{ \left( \frac{\norm{A'A}}{\norm{A'}\, \norm{A}}\right)^2}} \leq   d(\hv_1^\ast(A), \hv_2(A')) \leq  \frac{\norm{A'A}}{\norm{A'}\, \norm{A}} .
    \end{align*}
\end{lemma}

\begin{proof}
See~\cite[Propositions 2.23 and 2.24]{DK-book}.
\end{proof}

\begin{proposition}[Avalanche Principle]
\label{AP}
    There exist positive constants $c_i$, $i=0, 1, 2$ such that given $0<\kappa < c_0\epsilon^2$ and $A_0,\ldots, A_n\in \SL_2(\R)$, if
    \begin{enumerate}
        \item $\min_j\norm{A_j}^2 \geq \kappa^{-1}$;
        \item $\min_j\frac{\norm{A_jA_{j-1}}}{\norm{A_{j-1}\norm{A_j}}} \geq \epsilon$.
    \end{enumerate}
    Then, for $A^n := A_{n-1}\cdots\, A_0$,
    \begin{align*}
        \max\left\{
            d\left(
                \hv_1^\ast\left(
                    A^n
                \right),\,
                \hv_1^\ast\left(
                    A_{n-1}
                \right)
            \right),\,
            d\left(
                \hv_2\left(
                    A^n
                \right),\,
                \hv_2\left(
                    A_0
                \right)
            \right)
        \right\}\leq c_1\kappa\epsilon^{-1}.
    \end{align*}
    and
    \begin{align*}
        e^{-c_2\kappa\epsilon^{-1}n}
        \leq \frac{
        \norm{A_{n-1}\, \cdots A_1\, A_0}\,\norm{A_{1}}\, \cdots \, \norm{A_{n-2}} 
        }{
            \norm{A_1\, A_0}\cdots\, \norm{A_{n-1}\, A_{n-2}}
        } \leq 
        e^{c_2\kappa\epsilon^{-1}n}.
    \end{align*}
\end{proposition}
\begin{proof}
See~\cite[Proposition 2.42]{DK-book} or~\cite[Theorem 2.1]{DK-31CBM}.
\end{proof}

\begin{lemma}\label{190722.10}
    Given $\hat v, \hat w\in\bP^1$ and  $A_1,\ldots, A_n\in \SL_2(\R)$  assume that:
  \begin{enumerate}[label=(\alph*)]
    \item\label{item:200722.1} $\lambda \gg \gamma \gg t$;
    \item\label{item:200722.2} $\min_j\norm{A_j}  \geq e^{\lambda}$;
    \item\label{item:200722.3} $\min_j\frac{\norm{A_j\, A_{j-1}}}{\norm{A_{j-1}\norm{A_j}}} \geq e^{-\gamma}$;
    \item\label{item:200722.4}  $\min\left\{
            d(A_1\, \hv,\, \hw),\, d(A_1^{-1}\, \hw,\, \hv)
        \right\}\geq e^{-t}$;
    \item\label{item:200722.5}  $\min\left\{
              d(A_n\, \hv,\, \hw),\, d(A^{-1}_n\, \hw,\, \hv)
        \right\}\geq e^{-t}$.
    \item\label{item:200722.6} $d(A_n\, \hv,\, A^{-1}_1\, \hw) \geq e^{-t};$
    \end{enumerate}
 Then for all $j=1,\ldots, n-1$,
    \begin{enumerate}
        \item\label{item:310722.0} $d(\hv_1^\ast(A^n)\, \hv_2(A^n))\gtrsim e^{-t}$;
        
        \item\label{item:310722.1} $A^n$ is hyperbolic and $\lambda(A^n)\gtrsim e^{(\lambda - 2\gamma)n}$;
        
        \item\label{item:310722.2} $\norm{A^n\, v} \gtrsim e^{(\lambda - 2\gamma)n}$ and $\norm{(A^{n})^{-1}\, \hw}\gtrsim e^{(\lambda - 2\gamma)n}$;
    \end{enumerate}
\end{lemma}
\begin{proof}
    Using the conditions~\ref{item:200722.4}, \ref{item:200722.5} and Proposition~\ref{lem:080722.3} we have
    \begin{align*}
        d(A_n\, \hv,\, \hv_1^\ast(A_n)) \lesssim e^{-2\lambda + t}
        \quad
        \text{and}
        \quad
        d(A^{-1}_1\, \hw,\, \hv_2(A_1)) \lesssim e^{-2\lambda + t}.
    \end{align*}
    By the AP (Proposition~\ref{AP})
    with $\kappa:=e^{-2\,\lambda}$ and $\epsilon:=e^{-\gamma}$,
    \begin{align*}
        d(\hv_1^\ast(A^n),\, \hv_1^\ast(A_n)) \leq e^{-2\lambda +\gamma}
        \quad
        \text{and}
        \quad
        d(\hv_2(A^n),\, \hv_2(A_1)) \leq e^{-2\lambda + \gamma}.
    \end{align*}
    Applying triangular inequality with condition \ref{item:200722.6},
    \begin{align}\label{eq:220722.1}
        d(\hv_1^\ast(A^n)\, \hv_2(A^n))
        \gtrsim e^{-t} - 2e^{-2\lambda + \gamma} - 2e^{-2\lambda + t}
        \gtrsim e^{-t},
    \end{align}
    which give us Item \ref{item:310722.0}. The AP also implies that
    \begin{align} \nonumber
        \norm{A^n } &\gtrsim 
        \exp\left( - e^{-2\,(\lambda-\gamma)}\right)\, 
        \frac{
             \norm{A_{2}\, A_{1}}\cdots\, \norm{A_{n}\, A_{n-1}}
         }{
          \norm{A_{2}}\, \cdots \, \norm{A_{n-1}} 
         }\\
        &\gtrsim e^{-n\,\gamma}\,   \norm{A_{1}}\, \cdots \, \norm{A_{n}}
        \geq e^{(\lambda-\gamma)\, n} .
    \end{align}
    Item \ref{item:310722.1} follows from inequality \eqref{eq:220722.1} and Proposition~\ref{balanced radius} with $\lambda(A^n)\gtrsim e^{-t}\norm{A^n} \gtrsim e^{(\lambda - 2\gamma)n}$.
    
    Using the bounds above for $d(A^{-1}_1\, \hw,\, \hv_2(A_1))$, $d(\hv_2(A^n),\, \hv_2(A_1))$  and condition \ref{item:200722.4} we have
    \begin{align}\label{eq:220722.2}
        d(\hv,\, \hv_2(A^n)) \gtrsim e^{-t} - e^{-2\lambda + t} - e^{-2\lambda + \gamma}\gtrsim e^{-t}.
    \end{align}
    Hence, by Lemma \ref{LA-1},\, $\displaystyle \norm{A^n\, v}\gtrsim e^{(\lambda - \gamma)n - t}$. Using similar arguments we conclude that,\, $\displaystyle \norm{(A^n)^{-1}\, w} \gtrsim e^{(\lambda - 2\gamma) n}$ which proves item \ref{item:310722.2}.
    \end{proof}
\section{Appendix: derivative of projective actions}\label{Appendix:derivativeProjectiveActions}
We state and prove some general formulas for the derivatives of the action that will be used throughout this section.
Given a non-zero vector $w\in\R^2$, let $z(w)$ be the unique unit vector which makes $\{ w/\norm{w}, z(w)\}$  an orthonormal basis.

Given   $A\in \SL_2(\R)$, the derivative of its projective action $\hat A:\bP^1\to\bP^1$ is  
\begin{align}\label{eq:020622.6}
   D \hat A(\hat{w})v = \frac{
        \left(
            Aw\wedge\, Av
        \right)
    }{\norm{Aw}^2}\,z(Aw) .
\end{align}
If $w\in\hat w$ is a unit vector and $v\in T_{\hat w}\bP^1$ is a unit and positive tangent vector then $A\, w\wedge A\, v=w\wedge v =1$ and the norm of the derivative $D\hat A(\hat w)$ is equal to 
$\norm{A\, w}^{-2}$.

For a $C^1$  one parameter family $\{A_t\}_{t\in I}$ of $SL_2(\R)$ matrices, where $I\subseteq \R$ is an interval containing $0$,  
\begin{align}\label{eq:020622.7}
    \frac{d}{dt} A_t\hat{v}\biggr\rvert_{t=0} = \frac{
        ({A}_0v\wedge  \dot A_0v)
    }{\norm{A_0v}^2}z(A_0v) .
\end{align}

More generally, given a one  parameter family of cocycles
$A_t:X\to\SL_2(\R)$, defined in some interval $I\subseteq \R$,
\begin{proposition}
\label{winding derivative formula}
If $\bvec_j(t):=A^j_t(x)\, v/\|A^j_t(x)\, v\| $
for $j=0,1,\ldots, n$,
then
 $$ \frac{d}{dt} \frac{A^n_t(x)\, v}{\norm{A^n_t(x)\, v}} 
 =  \sum_{j=0}^{n-1}\frac{1}{\norm{A^{j}_t(T^{n-j} x)\, \bvec_{n-j}}^2} \,  \left( \bvec_{n-j}  \wedge  (\dot A_t\, A_t^{-1})(T^{n-j-1} x) \bvec_{n-j} \right) \, z(A^n_t \, v) $$
\end{proposition}

\begin{proof}
Since
$$  \frac{d}{dt} A^n_t(x) 
=\sum_{j=1}^n A^{n-j}_t(T^{j} x)\, \dot A_t(T^{j-1} x)\, A_t^{j-1}(x)  $$
we have
\begin{align*}
  \frac{d}{dt} \frac{A^n_t(x)\, v}{\norm{A^n_t(x)\, v}}  &=
  \sum_{j=1}^n \, \frac{A_t^n(x)\, v \, \wedge\, 
  A^{n-j}_t(T^{j} x)\, \dot A_t(T^{j-1} x)\, A_t^{j-1}(x)\, v}{\norm{A_t^n(x)\, v}^2}\, z(A^n_t\, v)\\
  &=
  \sum_{j=1}^n \, \frac{A_t^{j}(x)\, v \, \wedge\, 
  \dot A_t(T^{j-1} x)\, A_t^{j-1}(x)\, v}{\norm{A_t^n(x)\, v}^2}\, z(A^n_t\, v)\\
    &=
  \sum_{j=1}^n \,  \frac{\norm{A_t^{j-1}(x)\,v}^2\, A_t (T^{j-1} x)\, \bvec_{j-1} \, \wedge\, 
  \dot A_t(T^{j-1} x)\,\bvec_{j-1}}{\norm{A_t^n(x)\, v}^2}\, z(A^n_t\, v)\\
    &=
  \sum_{j=1}^n \, \frac{\norm{A_t^j(x)\,v}^2}{\norm{A_t^n(x)\, v}^2}\, \frac{  A_t (T^{j-1} x)\, \bvec_{j-1} \, \wedge\, 
  \dot A_t(T^{j-1} x)\,\bvec_{j-1}}{\norm{A_t(T^{j-1} x)\, \bvec_{j-1}}^2}\, z(A^n_t\, v)\\
    &=
  \sum_{j=1}^n \, \frac{1}{\norm{A_t^{n-j}(T^{j}x)\, \bvec_j }^2}\, \left( \bvec_{j} \, \wedge\, 
 ( \dot A_t\, A_t^{-1})(T^{j-1} x)\,\bvec_{j}  \right)\, z(A^n_t\, v)\\
  & =  \sum_{j=0}^{n-1}\frac{1}{\norm{A^{j}_t(T^{n-j} x)\, \bvec_{n-j}}^2} \,  \left( \bvec_{n-j}  \wedge  (\dot A_t\, A_t^{-1})(T^{n-j-1} x) \bvec_{n-j} \right) \, z(A^n_t \, v)
\end{align*}
\end{proof}

\begin{lemma}\label{lem:020822.1}
    Given a compact interval $I\subset \R$, there exist $C,\, c>0$ such that for every $x\in X$, $\hv\in \bP^1$ and $n\in \N$
    \begin{align*}
        \left|
            \frac{d}{dt}A^n_t(x)\, \hv
        \right|\leq C e^{cn}.
    \end{align*}
\end{lemma}
\begin{proof}
    Let $M_0 = \sup_{t\in I}\norm{\dot A_t\, A_t^{-1}}_{\infty}$ and $M_1 = \sup_{t\in I}\norm{A_t}_{\infty}$. Then, for every $\hw\in \bP^1$, $\norm{A^j(x)\, \hw} \geq M_1^{-j}$. Therefore, by Proposition \ref{winding derivative formula},
    \begin{align*}
        \left|
            \frac{d}{dt}A^n_t(x)\, \hv
        \right| \leq M_0\sum_{j=0}^{n-1}M_1^j \leq M_0 \leq \frac{M_0}{M_1-1}M_1^n = Ce^{cn},
    \end{align*}
    where $C = \frac{M_0}{M_1-1}$ and $c = \log M_1$.
\end{proof}

For Schr\"odinger cocycles more can be said.

\begin{lemma}
\label{Schrodinger winding}
Given a Schr\"odinger cocycle $A_E: X\to \SL_2(\R)$ with continuous potential $\phi:X\to \R$ and generated by the dynamical system $(X, T, \xi)$, there exists a constant $c_*>0$ such that for all $n\geq 2$,
all $x\in X$, all $E\in \R$ and all $\hat v\in\bP^1$,
\begin{align*}
    \left|
        \frac{d}{dE}A_E^n(x)\, \hv
    \right| = \frac{ A_E^n(x)\, v \, \wedge \, \frac{d}{dE} A_E^n(x)\, v}{\norm{A_E^n(x)\, v}^2} \, \geq \, c_* ,
\end{align*}
and
\begin{align*}
    \left|
        \frac{d}{dE}A_E^{-n}(x)\, \hv
    \right| = -\, \frac{ A_E^{-n}(x)\, v \, \wedge \, \frac{d}{dE} A_E^{-n}(x)\, v}{\norm{A_E^{-n}(x)\, v}^2} \, \geq \,   c_* .
\end{align*}
\end{lemma}

\begin{proof}
For $n=1$ a simple calculation gives
$$ \frac{ A_E(x)\, v \, \wedge \, \dot  A_E(x)\, v}{\norm{A_E(x)\, v}^2} =\frac{v_1^2}{\norm{A_E(x)\, v}^2}  $$
while using Proposition~\ref{winding derivative formula} with $n=2$
$$ \frac{ A_E^2(x)\, v \, \wedge \, \frac{d}{dE}  A_E^2(x)\, v}{\norm{A_E^2(x)\, v}^2} =   \frac{v_1^2}{\norm{A_E^2(x)\, v}^2} +
  \frac{((\phi(x)-E)\, v_1-v_2)^2}{\norm{A_E(T x)\, v}^2}  $$
  is positive and bounded away from $0$ because the denominators are bounded and the numerators add up to a positive definite quadratic form  $v_1^2 + ((\phi(x)-E)\, v_1-v_2)^2$, for any $x\in X$ and $E\in\R$.
 The general case follows also from Proposition~\ref{winding derivative formula} neglecting all terms but the last two.
 
 Finally, since
 $A_E^{-n}(x)=A_E^n(T^{-n}x)^{-1}$, if
 $w=A_E^{-n}( x)\, v$ \, then
 \begin{align*}
 A_E^{-n}(x)\, v \, \wedge \, \frac{d}{dE} A_E^{-n}(x)\, v
&= A_E^{-n}(x)\, v \, \wedge \,  A_E^{-n}(x)\,  \left( - \frac{d}{dE} A_E^{n}(T^{-n} x)\right)\,  A_E^{-n}(x)\, v\\ 
&=-  \, v \, \wedge \,   \left( \frac{d}{dE} A_E^{n}(T^{-n} x)\right)\,  A_E^{-n}(x)\, v\\ 
&=-  \, A_E^{n}(T^{-n} x)\, w \, \wedge \,    \frac{d}{dE} A_E^{n}(T^{-n} x)\,  w \, <\, 0 .
 \end{align*}
Therefore the projective curve $E\mapsto A_E^{-n}(x)\, \hat v$ winds in the opposite direction, and a  similar argument gives that $\left|\frac{d}{dE} A_E^{-n}(x)\, \hv \right|$   is bounded away from $0$.
\end{proof}

The space $\bP^1$ has a natural orientation. We say that a curve $\hat v:I\to\bP^1$ winds positively, resp. negatively, when it is positively, resp. negatively, oriented.
\begin{proposition}
\label{speed of stable and unstable directions}
If $M_E$ is hyperbolic with $\lambda(M_E)\geq \lambda_0>1$ for all $E$ in some compact interval $I$, then there exists a positive constant $K<\infty$ such that  the projective curves 
$E\mapsto \hat u(M_E)$ and $E\mapsto \hat s(M_E)$ wind around $\bP^1$ in opposite directions with non-zero speed bounded from above by
$K$.
The unstable curve $E\mapsto \hat u(M_E)$ winds positively, while
 the stable curve $E\mapsto \hat s(M_E)$ winds negatively.
 Moreover
\begin{align*}
 \frac{d}{dE}\hat{u}(M_E) = \frac{
    \lambda(E)\, ( u(E)  \wedge  \dot M_E\, u(E) )
}{ 
    \lambda(E)^2-1  
}\, z(\hat u(M_E))
\end{align*}
and
\begin{align*}
 \frac{d}{dE}\hat{s}(M_E) = -\, \frac{
    \lambda(E)\, ( s(E)  \wedge  \dot M_E^{-1}\, s(E) )
}{ 
    \lambda(E)^2-1  
}\, z(\hat s(M_E)).
\end{align*}
\end{proposition}

\begin{proof}
Breaking the interval $I$ into finitely many sub-intervals, if necessary, we can take $\hat e_1\in\bP^1$ such  that $\hat e_1\neq \hat s(M_E)$, for all $E\in I$. Hence $\hat u(M_E)= \lim_{k\to\infty} M_E^k\, \hat e_1$, with uniform convergence of the functions and their derivatives over the compact interval $I$.

As $k\to \infty$, the unit vectors
$\bvec_{k}(E):=\frac{M_E^k\, e_1}{\norm{M_E^k e_1}}$ converge geometrically and uniformly to an eigenvector $u(E)$
of $M_E$ in $\hat u(M_E)$ and, denoting by $\lambda(E)$ the absolute value of the corresponding eigenvalue, by Proposition~\ref{winding derivative formula}
we have
\begin{align*}
\frac{d}{dE} \hat u(M_E)&=  \lim_{k\to \infty}\sum_{j=0}^{k-1}  \frac{1}{\norm{M_E^{j}\, \bvec_{k-j}}^2}\,
\left(  \bvec_{k-j}  \wedge  \dot M_E\, M_E^{-1}\, \bvec_{k-j} \right)\, z(M_E^k \, \hat e_1)\\
&=\sum_{j=0}^\infty 
\frac{1}{\norm{M_E^{j}\, u(E)}^2}\,
\left(  u(E)  \wedge  \dot M_E\, M_E^{-1}\, u(E) \right)\, z(\hat u(M_E))\\
&=\sum_{j=0}^\infty 
\frac{1}{\lambda(E)^{2j+1}}\,
  u(E)  \wedge  \dot M_E\, u(E)  \, z(\hat u(M_E)) \\
  &= \frac{ \lambda(E)\, ( u(E)  \wedge  \dot M_E\, u(E) )}{ \lambda(E)^2-1  } \, z(\hat u(M_E)) .
\end{align*}
This concludes the argument for the unstable curve $\hat u(M_E)$.
The stable curve $\hat s(M_E)$ winds negatively because $\hat s(M_E)= \lim_{k\to\infty} M_E^{-k}\, \hat e_2$, for any vector $e_2$ such that $\hat e_2\neq \hat u(M_E)$, for all $E\in I$.  See Lemma~\ref{Schrodinger winding}.
The bound for the derivative of this stable curve is obtained in a similar way.
\end{proof}


\bibliographystyle{abbrv}
\bibliography{bib.bib}

\information

\end{document}